\documentclass[11pt]{amsart}
\pdfoutput=1
\usepackage[utf8]{inputenc}
\usepackage{amsmath, amssymb,amsthm}
\usepackage{tikz-cd}
\usepackage{tikz}
\usepackage[all,cmtip]{xy}
\usepackage{todonotes}
\usepackage{mathrsfs}
\usepackage{extarrows}
\usepackage{graphicx}
\usepackage{thmtools}
\usepackage{mathtools}
\usepackage{hyperref}
\usepackage{bbm}
\usepackage[top=1.50in, bottom=1.50in, left=1.25in, right=1.25in]{geometry}

\usepackage[english]{babel}

\usepackage[toc,page]{appendix}

\newcommand{\End}{\operatorname{End}}
\newcommand{\Id}{\operatorname{id}}

\newcommand{\Frob}{\operatorname{Frob}}
\newcommand{\can}{\operatorname{can}}
\newcommand{\Tr}{\operatorname{Tr}}

\newcommand{\et}{\operatorname{\acute{e}t}}
\newcommand{\Spec}{\operatorname{Spec}}
\newcommand{\Frac}{\operatorname{Frac}}
\newcommand{\red}{\operatorname{red}}
\newcommand{\perf}{\operatorname{perf}}
\newcommand{\af}{\operatorname{af}}
\newcommand{\Gal}{\operatorname{Gal}}
\newcommand{\GL}{\operatorname{GL}}
\newcommand{\Fl}{\mathcal{F}\!\ell}
\newcommand{\dep}{\operatorname{dep}}
\newcommand{\Hom}{\operatorname{Hom}}
\newcommand{\colim}{\operatorname{colim}}

\hypersetup{
    colorlinks,
    citecolor=blue,
    filecolor=blue,
    linkcolor=blue,
    urlcolor=blue
}

\renewcommand{\tilde}{\widetilde}

\newcommand{\SC}{\operatorname{sc}}

\numberwithin{equation}{section}
\newtheorem{theorem}[equation]{Theorem}
\newtheorem{prop}[equation]{Proposition}

\newtheorem{lem}[equation]{Lemma}

\theoremstyle{definition}

\newtheorem{remark}[equation]{Remark}

\newtheorem{example}[equation]{Example}

\begin{document}

\title[Mod $p$ geometric Satake]{Geometrization of the Satake transform \\ for mod $p$ Hecke algebras}
\date{}

\author{Robert Cass}
\address{University of Michigan, 530 Church St, Ann Arbor, MI, USA}
\email{cassr@umich.edu}

\author{Yujie Xu}
\address{M.I.T., 77 Massachusetts Avenue,
Cambridge, MA, USA}
\email{yujiexu@mit.edu}

\maketitle

\begin{abstract}
    We geometrize the mod $p$ Satake isomorphism of Herzig and Henniart--Vign\'eras using Witt vector affine flag varieties for reductive groups in mixed characteristic.~We deduce this as a special case of a formula, stated in terms of the geometry of generalized Mirkovi\'c--Vilonen cycles, for the Satake transform of an arbitrary pararhoric mod $p$ Hecke algebra with respect to an arbitrary Levi subgroup.
    Moreover, we prove an explicit formula for the convolution product in an arbitrary parahoric mod $p$ Hecke algebra.
    Our methods involve the constant term functors inspired from the geometric Langlands program, and we also treat the case of reductive groups in equal characteristic. We expect this to be a first step towards a geometrization of a mod $p$ Local Langlands Correspondence.
\end{abstract}

\tableofcontents

\section{Introduction}
\addtocontents{toc}{\protect\setcounter{tocdepth}{0}}
Let $G$ be a reductive algebraic group defined over a nonarchimedean local field $F$ with finite residue field of characteristic $p$. Fix a coefficient field $k$. For every compact open subgroup $K \subset G(F)$, let $\mathcal{H}_K$ be the Hecke algebra of compactly supported functions $K \backslash G(F) / K \rightarrow k$. The Hecke algebras $\mathcal{H}_K$ arise naturally in the study of smooth representations of $G(F)$ on $k$-vector spaces (see \S \ref{sect-context}). 

In the present paper, we focus on the case where $k$ has characteristic $p$ equal to that of the residue field of $F$. Motivated by recent progress in the mod $p$ Langlands program, as well as recent geometrization programs of Langlands Correspondences, we apply the geometry of Witt vector affine flag varieties to obtain results on mod $p$ Hecke algebras. Our methods are inspired by the geometric Langlands program, especially the geometric Satake equivalence \cite{Lusztig:Singularities, Ginzburg:Satake, MirkovicVilonen:Geometric} and its analogue in mixed characteristic \cite{Zhu:Mixed}.
When $K$ is a special parahoric subgroup, we recover the mod $p$ Satake isomorphism of Herzig \cite{Herzig:Satake} and Henniart--Vign\'eras \cite{HenniartVigneras:Satake}. The Satake isomorphism in these works is the mod $p$ counterpart of the Satake isomorphism of Haines--Rostami \cite{HainesRostami:Satake} in characteristic zero. Our methods also yield formulas for the inverse of the mod $p$ Satake isomorphism in \cite{Herzig:Classification, Ollivier:Inverse, InverseSatake}.

\subsection{Main results}
Let $A$ be a maximal $F$-split torus in $G$ and let $\mathbf{f}$ be a facet in the apartment $\mathscr{A}(G,A,F)$ in the enlarged building. Let $\mathcal{O}_F$ be the ring of integers of $F$ and let $\mathcal{G}$ be the parahoric $\mathcal{O}_F$-group scheme such that $K:=\mathcal{G}(\mathcal{O}_F)$ is the connected fixer of $\mathbf{f}$. Then $K$ is a compact open subgroup of $G(F)$, and we can form the mod $p$ Hecke algebra $\mathcal{H}_K$ of compactly supported functions $K \backslash G(F) / K \rightarrow \mathbb{F}_p$. The algebra structure is given by the convolution product $*$ (see \eqref{eqn-convolution}).

\subsubsection{Convolution}
Let $\mathbb{F}_q$ be the residue field of $F$. If $F$ has characteristic zero, the Witt vector affine flag variety $\Fl_{\mathcal{G}}$, first constructed in \cite{Zhu:Mixed}, is an increasing union of perfections of projective $\mathbb{F}_q$-schemes \cite{BhattScholze:Projectivity} such that $\Fl_{\mathcal{G}}(\mathbb{F}_q) = G(F)/K$. The closures of the left $K$-orbits in $\Fl_{\mathcal{G}}$, called perfect Schubert schemes, are enumerated by certain double cosets $W_{\mathbf{f}} \backslash W^{\sigma} / W_{\mathbf{f}}$ (see \S \ref{sect-shubert-schemes}). Such Schubert schemes appear as perfections of irreducible components of the mod $p$ fibers of local models of Shimura varieties. If $F$ has characteristic $p$, for uniformity of exposition we let $\Fl_{\mathcal{G}}$ be the perfection of the usual (power series) affine flag variety associated to $G$.

For $w \in W_{\mathbf{f}} \backslash W^{\sigma} / W_{\mathbf{f}}$, let $(\underline{\mathbb{F}}_p)_w$ be the constant sheaf on the associated perfect Schubert scheme $\Fl_w$. The function-sheaf dictionary associates $\{(\underline{\mathbb{F}}_p)_w\}_w$ to a certain basis $\{\phi_w\}_w$ of $\mathcal{H}_K$. Specifically, $\phi_w$ is the sum of the characteristic functions of the double cosets in $K \backslash G(F) / K $ associated to elements in $W_{\mathbf{f}} \backslash W^{\sigma} / W_{\mathbf{f}}$ bounded by $w$ in the Bruhat order. 

To state our first theorem, we need the convolution map $m \colon \Fl_{\mathcal{G}} \times^{L^+\mathcal{G}} \Fl_{\mathcal{G}} \rightarrow \Fl_{\mathcal{G}}$ (see \eqref{eqn-convolution-diagram}) arising from the multiplication map on $G(F)$. Here $L^+\mathcal{G}$ is the positive loop group (see \eqref{eqn-loop-groups}), which is a perfect ind-scheme such that $L^+\mathcal{G}(\mathbb{F}_q) = K$. The sheaf functor $Rm_!$ corresponds to the convolution product $*$ for the Hecke algebra $\mathcal{H}_K$ under the function-sheaf dictionary (Lemma \ref{lem-conv-points}, Lemma \ref{lem-function-sheaf}).

\begin{theorem}[Theorem \ref{thm-conv-formula}] \label{mainthm-conv-formula}
Let $w_1, w_2 \in W_{\mathbf{f}} \backslash W^\sigma / W_{\mathbf{f}}$, and let $w \in  W_{\mathbf{f}} \backslash W^{\sigma} / W_{\mathbf{f}}$ be such that $m(\Fl_{w_1} \times^{L^+\mathcal{G}} \Fl_{w_2}) = \Fl_w$.
Then
$$\phi_{w_1} * \phi_{w_2} = \phi_w.$$
\end{theorem}

In particular, the convolution of two elements in $\{\phi_w\}_w$ is again an element in $\{\phi_w\}_w$. For a given pair $(w_1, w_2)$, the element $w$ may be computed explicitly by combinatorial computations in the affine Weyl group. The key geometric input in the proof of Theorem \ref{mainthm-conv-formula} is the Demazure scheme $\pi_{\dot{w}} \colon D_{\dot{w}} \rightarrow \Fl_w$  \eqref{eqn-Demazure-def} for possibly non-reduced $\dot{w}$. This allows one to show that the fibers of $m$ are built out of Schubert varieties, which in particular have the property that their number of $\mathbb{F}_q$-points is congruent to $1 \pmod{p}$ (see the proof of Theorem \ref{thm-convolution}).

\subsubsection{Satake transform}
Now let $M$ be a Levi subgroup of $G$ given by the centralizer of a subtorus of $A$. To the facet $\mathbf{f}$, we may naturally associate a facet in the apartment $\mathscr{A}(M,A,F)$ whose parahoric group scheme $\mathcal{M}$ satisfies $\mathcal{M}(\mathcal{O}_F) = M(F) \cap K$ \cite[A.2]{Richarz:Affine}. Let $P$ be a parabolic subgroup of $G$ with Levi factor $M$ and unipotent radical $U$. When $K$ is special, Herzig \cite{Herzig:Satake} and Henniart--Vign\'eras \cite{HenniartVigneras:Satake} defined a Satake transform\footnote{The formula for $\mathcal{S}$ makes sense even when $K$ is non-special, although it is not a homomorphism in this generality.}
$$    \mathcal{S} \colon \mathcal{H}_{K} \rightarrow \mathcal{H}_{M(F) \cap K}, \quad \mathcal{S}(f)(m) = \sum_{u \in U(F) / U(F) \cap K} f(mu), \: \: m \in M(F).$$

Our second result is an explicit formula for $\mathcal{S}$ in the basis $\{\phi_w\}_w$ where $K$ is any parahoric subrgoup. To state it, note that there exists a cocharacter $\lambda \colon \mathbb{G}_m \rightarrow G$ such that $M= G^0$ is the group of \textit{fixed points} under the corresponding conjugation action of $\mathbb{G}_m$, and $P=G^+$ is the \textit{attractor} for this $\mathbb{G}_m$-action in the sense of \eqref{eqn-Gm-action}. For an action of the perfection $(\mathbb{G}_m)_{\perf}$ on a perfect $\mathbb{F}_q$-scheme, we use the same notation to denote the fixed points and attractors as functors on perfect $\mathbb{F}_q$-schemes.
Then there is an induced action of $(\mathbb{G}_m)_{\perf}$ on $\Fl_{\mathcal{G}}$ such that $\Fl_{\mathcal{M}} \subset (\Fl_{\mathcal{G}})^0$. Here $(\Fl_{\mathcal{G}})^0$ decomposes as a disjoint union of connected components, and $\Fl_{\mathcal{M}}$ is a union of some subset of these components. The natural map $q^+ \colon (\Fl_{\mathcal{G}})^+ \rightarrow (\Fl_{\mathcal{G}})^0$, which is informally $x \mapsto \lim_{t \rightarrow 0} \lambda(t) \cdot x \cdot \lambda(t)^{-1}$, induces a bijection on connected components.

The geometrically connected components of $\Fl_{\mathcal{M}}$ are indexed by a certain group $\pi_1(M)^\sigma_I$ (see~\S\ref{sect-shubert-schemes}).~For $c \in \pi_1(M)^\sigma_I$ and $w \in W_{\mathbf{f}} \backslash W^\sigma / W_{\mathbf{f}}$, let $(\Fl_{\mathcal{M}})^c$ be the corresponding connected component of $\Fl_{\mathcal{M}}$. Let
$\phi_{c,w} \in \mathcal{H}_{\mathcal{M}}$ be the characteristic function of the intersection $(\Fl_{\mathcal{M}})^c(\mathbb{F}_q) \cap \Fl_w(\mathbb{F}_q)$. Here $(\Fl_{\mathcal{M}})^c\cap \Fl_w$ is a (geometrically) connected component of $(\Fl_w)^+$, i.e.~an attractor, and the irreducible components of this intersection are generalized versions of the \textit{Mirkovi\'c--Vilonen cycles} in \cite[Theorem 3.2]{MirkovicVilonen:Geometric}.

\begin{theorem}[Theorem \ref{thm-satake-formula}] \label{mainthm-satake-formula}
The Satake transform $\mathcal{S} \colon \mathcal{H}_{K} \rightarrow \mathcal{H}_{M(F)\cap K}$ satisfies
$$
\mathcal{S}(\phi_w)=
\begin{cases}
\phi_{c,w}, & c \in \pi_1(M)^\sigma_I \text{ is such that } (\Fl_{\mathcal{M}})^c \cap \Fl_{w} \text{ is closed in } \Fl_{w} \\
0, & \text{ for all } c \in \pi_1(M)^\sigma_I, \: (\Fl_{\mathcal{M}})^c \cap \Fl_{w} \text{ is not closed in } \Fl_{w}.
\end{cases}$$
\end{theorem}
Over an algebraic closure, $(\Fl_{w, \overline{\mathbb{F}}_q})^+$ has a unique closed attractor (Theorem \ref{prop-equiv-res}), so the element $c$ in Theorem \ref{mainthm-satake-formula} is unique if it exists.
Furthermore, the element $c$ is characterized by the property that $q^+$ restricts to an isomorphism from the closed attractor onto its image. If $K$ is special there is always such an element $c$, but in general the inclusion $\Fl_{\mathcal{M}} \subset (\Fl_{\mathcal{G}})^0$ is strict. Informally, the idea behind the proof of Theorem \ref{mainthm-satake-formula} is that point-counting $(\text{mod } p)$ on Mirkovi\'c--Vilonen cycles may be performed on $D_{\dot{w}}$. Since the map $(D_{\dot{w}})^+ \rightarrow (D_{\dot{w}})^0$ is a disjoint union of perfected affine bundles, the fibres (when positive-dimensional) have vanishing compactly supported mod $p$ \'etale cohomology. This makes the point-counting more accessible.

\subsubsection{Special parahorics}
Now assume $M = C_G(A)$ is a minimal Levi subgroup and $\mathbf{f}$ is a special vertex. Then $\Lambda : = M(F)/(M(F) \cap K)$ is a finitely generated abelian group, and $\mathcal{H}_{M(F) \cap K}$ is canonically isomorphic to the group algebra $\mathbb{F}_p[\Lambda]$ (see \cite[\S 6]{HenniartVigneras:Satake} and \cite[\S 11]{HainesRostami:Satake} for more details).
For $z \in \Lambda$, let $e^z \in \mathbb{F}_p[\Lambda]$ be the associated element. 
The choice of $P$ determines a set of anti-dominant representatives $\Lambda_- \subset \Lambda$ for the orbits of the action of the finite Weyl group $W(G,A)$. Moreover, there is a canonical bijection of sets $\Lambda_- \cong W_{\mathbf{f}} \backslash W^{\sigma} / W_{\mathbf{f}}$. In this case, Theorem \ref{mainthm-conv-formula} and Theorem \ref{mainthm-satake-formula} recover the following result, originally due to Herzig \cite[Theorem 1.2]{Herzig:Satake} when $K$ is hyperspecial, and Henniart--Vign\'eras \cite[\S 1.5]{HenniartVigneras:Satake} when $K$ is special. Moreover, we recover the explicit formulas in \cite[Proposition 5.1]{Herzig:Classification} and \cite[Theorem 5.5]{Ollivier:Inverse} (where $G$ is split and $K$ is hyperspecial), and in \cite[Theorem 1.1]{InverseSatake} (where $G$ is arbitrary and $K$ is special).

\begin{theorem}[Theorem \ref{thm-special-satake}] \label{mainthm-special}
If $K$ is special, for $z_1, z_2 \in \Lambda_-$, we have $$m(\Fl_{z_1} \times^{L^+\mathcal{G}} \Fl_{z_2}) = \Fl_{z_1z_2} \quad \text{and} \quad \phi_{z_1} * \phi_{z_2} = \phi_{z_1z_2}.$$
In particular, $\mathcal{H}_{K}$ is commutative. Moreover, the Satake transform $\mathcal{S} \colon  \mathcal{H}_{K} \rightarrow \mathcal{H}_{M(F) \cap K}$ is injective, identifies $\mathcal{H}_{K}$ with $\mathbb{F}_p[\Lambda_-]$, and satisfies
$$\mathcal{S}(\phi_z) = e^{z}, \quad \text{for all } z \in \Lambda_-.$$
\end{theorem}

In comparison, when $k= \mathbb{C}$ and $K$ is special, the Hecke algebra $\mathcal{H}_{K}$ is also commutative  \cite[Theorem 1.0.1]{HainesRostami:Satake}, and the Satake transform identifies it with the $W(G,A)$-invariants in $\mathbb{C}[\Lambda]$. Once the geometric preliminaries in Theorem \ref{mainthm-conv-formula} and Theorem \ref{mainthm-satake-formula} are established, Theorem \ref{mainthm-special} follows from the identification $\Lambda = \pi_1(M)_I^\sigma$ \cite[Proposition 1.0.2]{HainesRostami:Satake}, and
the following combinatorial fact \cite[\S 6.9]{HenniartVigneras:Satake}:
\begin{equation} \label{eqn-special-intersection} KzK \cap zU(F) = z(U(F) \cap K), \quad z \in \Lambda_-.
\end{equation}

To make the connection with Theorem \ref{mainthm-satake-formula} explicit, note that when $z \in \Lambda_-$ (which may be viewed as a subset of both $\pi_1(M)_I^\sigma$ and $W_{\mathbf{f}} \backslash W^{\sigma} / W_{\mathbf{f}}$) we have $\phi_{z,z} = e^z$. The unique closed attractor in $(\Fl_z)^+$ is the single $\mathbb{F}_q$-point corresponding to $z \in \Lambda_-$.
Moreover, if $P \subset P' \subset G$ is any intermediate parabolic with Levi factor $M'
$, then $\pi_1(M')_I^\sigma$ is naturally identified with a submonoid of $\Lambda$ containing $\Lambda_-$.
For the $(\mathbb{G}_m)_{\perf}$-action induced by a cocharacter such that $P' = G^+$, the unique closed attractor in $(\Fl_z)^+$ is $(\Fl_{\mathcal{M}'})^z \cap \Fl_{z}$.

\begin{remark}
Let $V$ be a finite-dimensional representation of $K$ over a field $k$ of characteristic $p>0$. Associated to $V$ is the Hecke algebra $\mathcal{H}_K(V)$ of compactly supported functions $G(F) \rightarrow \End_k(V)$ which intertwine the left and right actions of $K$. Henniart--Vign\'eras \cite[\S 1.8]{HenniartVigneras:Satake} established a Satake isomorphism for $\mathcal{H}_K(V)$ similar to Theorem \ref{mainthm-special}. When $V=\mathbb{F}_p$ is the trivial representation, we have $\mathcal{H}_K = \mathcal{H}_K(V)$. It is an interesting question to geometrize these Hecke algebras $\mathcal{H}_K(V)$ with non-trivial weights as well.
\end{remark}

\subsection{Background} \label{sect-context}
We first comment on the appearance of Hecke algebras in the Local Langlands program, which predicts a relationship between smooth representations of $G(F)$ and representations of the Weil group of $F$ in the Langlands dual group ${}^LG$. An important tool in the study of smooth representations of $G(F)$ is the functor of $K$-invariants $V \mapsto V^K$, where $V$ is a representation of $G(F)$ over a field $k$ and $K \subset G(F)$ is a compact open subgroup. The space $V^K$ is naturally a module over the Hecke algebra $\mathcal{H}_K$. When $k = \mathbb{C}$, the functor $V \mapsto V^K$ is an equivalence between irreducible smooth representations of $G(F)$ admitting a $K$-fixed vector, and irreducible modules over $\mathcal{H}_K$. When $k$ has characteristic $p > 0$ the functor $V \mapsto V^K$ is no longer an equivalence. Nonetheless, the study of mod $p$ Hecke algebras plays an essential role in the mod $p$ Langlands program. 
 
 For example, mod $p$ Hecke algebras are used in the classification of irreducible admissible mod $p$ representations of $G(F)$ in terms of supercuspidals \cite{AHHV-classification, HV-19}.~Little is known about mod $p$ supercuspidal representations outside of a few small rank cases \cite{Breuil:Modulaires, Abdellatif:SL2, Koziol:U2}. However, numerical evidence \cite{Vigneras:Pro, Ollivier:Parabolic}, later upgraded to a functor \cite{GK1, GK2, GK3}, suggests a correspondence between supersingular modules over the pro-$p$ Iwahori Hecke algebra and mod $p$ Galois representations.
 
 The results in this paper are a first step toward a geometrization program of a mod $p$ Local Langlands Correspondence. We expect that our methods generalize to the pro-$p$ Iwahori Hecke algebra, which is geometrized by sheaves on the affine flag variety of a non-parahoric subgroup (the pro-unipotent radical of the Iwahori loop group). Moreover, our affine flag varieties are the special fibers of Beilinson--Drinfeld Grassmannians arising in $p$-adic geometry \cite{ScholzeWeinstein:Berkeley}. In this direction, Mann \cite{Mann:6} has developed a theory of $p$-adic \'etale sheaves which is expected to be useful in a $p$-adic version of the ground-breaking work of Fargues--Scholze \cite{FarguesScholze:Geometrization}. We thus expect a fruitful interplay between our work and the $p$-adic Langlands correspondence \cite{Colmez:GL2}, including groups beyond $\GL_2(\mathbb{Q}_p)$. We also expect a connection with \cite{Pepin-Schmidt-KL}, which relates the Emerton--Gee stack \cite{EmertonGee:Stack} to a Kazhdan--Lusztig style parametrization (as in \cite{KL:Deligne}) of mod $p$ Hecke modules for $\GL_2(\mathbb{Q}_p)$.

\subsubsection{Geometric Satake}
Our methods are inspired by Zhu's geometric Satake equivalence \cite{Zhu:Mixed}, where $F$ has characteristic zero.~This relates 
the category $P_{L^+\mathcal{G}}(\Fl_{\mathcal{G}})$ of $L^+\mathcal{G}$-equivariant perverse $\overline{\mathbb{Q}}_\ell$-sheaves on $\Fl_{\mathcal{G}}$ to 
algebraic representations of ${}^LG$ on finite-dimensional $\overline{\mathbb{Q}}_\ell$-vector spaces. The function-sheaf dictionary provides the connection with Hecke algebras by associating to each $\mathcal{F} \in P_{L^+\mathcal{G}}(\Fl_{\mathcal{G}})$ a function $\mathcal{F}^{\Tr} \colon K \backslash G(F) /K \rightarrow \overline{\mathbb{Q}}_\ell$. The convolution map $m \colon \Fl_{\mathcal{G}} \times^{L^+\mathcal{G}} \Fl_{\mathcal{G}} \rightarrow \Fl_{\mathcal{G}}$ encodes the convolution product $*$ in the following way \cite[Lemma 5.6.1]{Zhu:Intro}: for $\mathcal{F}_1$, $\mathcal{F}_2 \in P_{L^+\mathcal{G}}(\Fl_{\mathcal{G}})$, there is a perverse sheaf $\mathcal{F}_1 \tilde\boxtimes \mathcal{F}_2$ on $\Fl_{\mathcal{G}} \times^{L^+\mathcal{G}} \Fl_{\mathcal{G}}$ which satisfies $(Rm_!(\mathcal{F}_1 \tilde\boxtimes \mathcal{F}_2))^{\Tr} = \mathcal{F}_1^{\Tr} * \mathcal{F}_2^{\Tr}$. 

The papers \cite{Cass:Perverse, CP:Constant} established a similar story for perverse $\mathbb{F}_p$-sheaves on the affine Grassmannian, where $F$ has characteristic $p$ (again for $K$ hyperspecial). Notably, a monoid related to $\Spec(\mathcal{H}_K)$ appears rather than ${}^LG$, and the simple perverse $\mathbb{F}_p$-sheaves turn out to be constant sheaves supported on Schubert varieties. This reflects the nature of the singularities of these Schubert varieties \cite[Theorem 1.4, Theorem 1.7]{Cass:Perverse}, namely, their seminormalizations are Cohen-Macaulay and $F$-rational (see also \cite[Theorem 4.1]{FHLR:Local}). No appropriate analogue of these properties is currently known for Witt vector Schubert schemes, so in this paper we work directly with constant sheaves rather than perverse sheaves. Our results also generalize the mod $p$ Hecke algebra results in \cite{Cass:Perverse, CP:Constant} (where $F$ has equal characteristic) to non-split groups and arbitrary parahoric subgroups. Due to the lack of \emph{smooth} $\mathbb{G}_m$-equivariant deperfections of Witt vector Demazure schemes, we cannot apply all the foundational results of \cite[\S2]{CP:Constant}, cf. the proof of Theorem \ref{prop-equiv-res}.

In another direction, the papers \cite{RicharzScholbach:Satake, RicharzScholbach:Witt} (for rational coefficients) and \cite{CassvdHScholbach:MotivicSatake} (for integral coefficients and $F$ of equal characteristic) gave a motivic refinement of the geometric Satake equivalence for split reductive groups. The latter provides a geometrization of the generic hyperspecial Hecke algebra, i.e.~the Hecke algebra with coefficients in $\mathbb{Z}[\mathbf{q}]$, where $\mathbf{q}$ is an indeterminate (see \cite[\S 6.3]{CassvdHScholbach:MotivicSatake}). Specialization along the quotient map $p \mapsto \mathbf{q} \mapsto 0$ recovers the hyperspecial mod $p$ Hecke algebra. It would be interesting to generalize \cite{CassvdHScholbach:MotivicSatake} to $F$ of mixed characteristic and to find a realization functor from motivic sheaves to the mod $p$ \'etale sheaves in the current paper which lifts the specialization map on Hecke algebras.

\subsubsection{Constant terms}
Recall the parabolic with Levi decomposition $P = MU$ arising from the cocharacter $\lambda$. Associated to $K$ there are integral models $\mathcal{P}$ of $P$ and $\mathcal{M}$ of $M$. The natural maps $\mathcal{M} \leftarrow \mathcal{P}$, $\mathcal{P} \rightarrow \mathcal{G}$ and the fixed points $(\Fl_{\mathcal{G}})^0 $ and attractors $(\Fl_{\mathcal{G}})^+$ for the resulting $(\mathbb{G}_m)_{\perf}$-action on $\Fl_{\mathcal{G}}$ are related by the following commutative diagram.
$$\xymatrix{
\Fl_{\mathcal{M}} \ar[d]_{p^0} & \ar[l]_q \Fl_{\mathcal{P}} \ar[d]^{p^+} \ar[r]^{\iota} & \Fl_{\mathcal{G}} \ar[d]^{\Id} \\
(\Fl_{\mathcal{G}})^0 & \ar[l]_{q^+} (\Fl_{\mathcal{G}})^+ \ar[r]^{\iota^+}  & \Fl_{\mathcal{G}}
}$$
The maps $p^0$ and $p^+$ are open and closed immersions \cite[Theorem 5.2]{AGLR:Local}, so they are inclusions of some connected components. The sheaf functor $Rq_! \circ \iota^*$ is called a \textit{constant term functor} in the geometric Langlands program, and it corresponds to the Satake transform
under the function-sheaf dictionary (Theorem \ref{thm-satake-k-points}). In particular, the fibers of $q$ are $U(F)$-orbits. The term ``constant term functor'' reflects the fact that $Rq_! \circ \iota^*$ is a local analogue of a functor, obtained by replacing $\Fl_{\mathcal{G}}$ with the moduli stack of $G$-bundles on a global curve in positive characteristic, which encodes the constant term operation on automorphic functions under the function-sheaf dictionary \cite{GaitsgoryBraverman:Eisenstein, GaitsgoryDrinfeld:Constant}. In the text, we work only with the bottom row of this diagram, which is sufficient since $Rq_! \circ \iota^*$ is a direct summand of $Rq^+_! \circ (\iota^+)^*$. Furthermore, because the $*$-pullback of a constant sheaf is constant, we will focus our attention on the functor $Rq^+_!$.
\\ \\
\noindent\textbf{Acknowledgements.} The authors would like to thank João Lourenço for comments on an earlier version of this paper and for helpful discussions on deperfections of perfect Schubert schemes. The authors would also like to thank Tom Haines, Claudius Heyer, Thibaud van den Hove, and Timo Richarz for helpful comments. R.C.~was supported by NSF grant DMS~2103200. Y.X.~was supported by NSF grant DMS~2202677. The authors would like to thank the University of Michigan, MIT and Harvard for providing wonderful working environments.

\addtocontents{toc}{\protect\setcounter{tocdepth}{2}}
\section{Geometric results} \label{sect-Geometric-results}
\subsection{Witt vector affine flag varieties}
\subsubsection{Notations}
Let $F$ be a complete discrete valuation field with ring of integers $\mathcal{O}_F$ and perfect residue field $k$ of characteristic $p>0$. We assume that $k$ is a finite field or an algebraic closure thereof. Fix a uniformizer $\varpi \in \mathcal{O}_F$. If $F$ has characteristic zero, for a $k$-algebra $R$ let $W(R)$ be its ring of Witt vectors, and let $W_{\mathcal{O}_F}(R) = W(R) \otimes_{W(k)} \mathcal{O}_F$. Let $W_{\mathcal{O}_F, n}(R) = W(R) \otimes_{W(k)} \mathcal{O}_F/ \varpi^n$. If $F$ has characteristic $p$, the choice of $\varpi$ induces an isomorphism $F \cong k ( \!(t) \!)$. 
For a $k$-scheme $X$, we denote its perfection by
$$X_{\perf}:= \lim (\cdots \xrightarrow{\text{Fr}} X \xrightarrow{\text{Fr}} X ), $$ where $\text{Fr}$ is the absolute Frobenius morphism. The underlying topological spaces of $X$ and $X_{\perf}$ are canonically isomorphic. Moreover, there is a natural equivalence of small \'etale sites 
$X_{\et} \cong (X_{\perf})_{\et}$ \cite[Tag 04DY]{stacks-project}. In particular, the \'etale cohomology groups of $X$ and $X_{\perf}$ are canonically isomorphic.

\subsubsection{The affine flag variety}
Let $G$ be a connected reductive group over $F$ and let $\mathscr{B}(G,F)$ be its Bruhat--Tits building.
For each facet $\mathbf{f} \subset \mathscr{B}(G,F)$, let $\mathcal{G} := \mathcal{G}_{\mathbf{f}}$ be the parahoric $\mathcal{O}_F$-group scheme with generic fiber $G$ as in \cite[D\'efinition 5.2.6 ff.]{BruhatTits:II}. Let $K := \mathcal{G}(\mathcal{O}_F)$. If $F$ has characteristiz zero, following \cite{Zhu:Mixed} we define the following functors on \emph{perfect} $k$-algebras:
\begin{equation} \label{eqn-loop-groups} LG \colon R \mapsto G(W_{\mathcal{O}_F}(R)[1/p]), \quad L^+\mathcal{G} \colon R \mapsto \mathcal{G}(W_{\mathcal{O}_F}(R)).\end{equation} The functor $LG$ is the called the loop group, and $L^+\mathcal{G}$ is called the positive loop group. We also define the functor
$L^n\mathcal{G} \colon R \mapsto \mathcal{G}(W_{\mathcal{O}_F, n}(R)).$ Then $L^+\mathcal{G} = \displaystyle\lim_{\longleftarrow} L^n\mathcal{G}$.
The affine flag variety of $\mathcal{G}$ is the functor on perfect $k$-algebras given by the \'etale-quotient
$$\Fl_{\mathcal{G}} = LG/L^+\mathcal{G}.$$
By \cite{BhattScholze:Projectivity}, $\Fl_{\mathcal{G}}$ is represented by an increasing union of perfections of projective $k$-schemes.

If $F$ has characteristic $p$, the loop group (resp. positive loop group) is $LG(R) = G(R(\!(t)\!))$ (resp. $L^+\mathcal{G}(R) = \mathcal{G}(R[\![t]\!])$. We set $L^n\mathcal{G}(R) = G(R[t]/t^n)$. These functors are usually defined on all $k$-algebras, but for uniformity of exposition we restrict to perfect $k$-algebras. Then the \'etale-quotient $\Fl_{\mathcal{G}} = LG/L^+\mathcal{G}$ is represented by the perfection of the ind-projective $k$-scheme constructed in \cite{PappasRapoport:Twisted}.

\subsubsection{Greenberg functors}
If $F$ has characteristic zero, there is no canonical ind-projective $k$-scheme whose perfection is $\Fl_{\mathcal{G}}$. The problem is that the natural extension of the functor $LG$ to non-perfect $k$-algebras is not well-behaved. (For example, if $R$ is non-perfect, $W_{\mathcal{O}_F}(R)$ can have $p$-torsion, in which case $W_{\mathcal{O}_F}(R)$ is not a submodule of $W_{\mathcal{O}_F}(R)[1/p]$.) However, as observed by Greenberg \cite{Greenberg:Schemata}, the functor $L^+\mathcal{G}$ is well-behaved on all $k$-algebras. Following \cite{Zhu:Mixed}, we define the following functors on all $k$-algebras:
$$L^+_p\mathcal{G} \colon R \mapsto G(W_{\mathcal{O}_F}(R)), \quad L^n_p\mathcal{G} \colon R \mapsto \mathcal{G}(W_{\mathcal{O}_F,n}(R)), \: n \geq 1.$$
Then $L^+_p\mathcal{G} = \displaystyle\lim_{\longleftarrow} L^n_p\mathcal{G}$ and $L^+\mathcal{G} = (L^+_p\mathcal{G})_{\perf}$. The functor $L^n_p\mathcal{G}$ is represented by a $k$-scheme of finite-type. 
If $F$ has characteristic $p$, let $L_pG$, $L^+_p\mathcal{G}$, and $L^n_p \mathcal{G}$ be the extensions of  $LG$, $L^+\mathcal{G}$, and $L^n \mathcal{G}$ to all $k$-algebras. There are no issues with $L_pG$ because $R[\![t]\!]$ never has $t$-torsion. We will use these functors to construct \textit{deperfections} of finite-dimensional subschemes of $\Fl_{\mathcal{G}}$, i.e.~finite-type $k$-schemes whose perfections are subschemes of $\Fl_{\mathcal{G}}$.

\subsubsection{Rational points}
Let $\breve{F}$ be the completion of a maximal unramified extension of $F$ and let $\overline{k}$ be its residue field. Since $\overline{k}$ is local for the \'etale topology,
$$\Fl_{\mathcal{G}}(\overline{k}) = LG(\overline{k})/L^+\mathcal{G}(\overline{k}) = G(\breve{F})/\mathcal{G}(\mathcal{O}_{\breve{F}}).$$

\begin{lem} \label{lem-k-points}
We have
$\Fl_{\mathcal{G}}(k) = G(F)/\mathcal{G}(\mathcal{O}_F)$.
\end{lem}

\begin{proof}
This follows from the fact that $H^1(\Gal(\overline{k}/k), \mathcal{G}(\mathcal{O}_{\breve{F}})) = 0$ \cite[Lemma 8.1.4]{KalehtaPrasad:Bruhat} and the exact sequence in Galois cohomology \cite[Proposition 36]{Serre:Galois}.
\end{proof}

\subsubsection{Perfect Schubert schemes} \label{sect-shubert-schemes} Choose a maximal $F$-split torus $A$ such that $\mathbf{f}$ is contained in the apartment $\mathscr{A}(G,A,F)$. 
Let $\mathbf{f}' \subset \mathscr{A}(G,A,F)$ be another facet. 
For $w \in L^+ \mathcal{G}_{\mathbf{f}'}(k) \backslash LG(k) /  L^+ \mathcal{G}_{\mathbf{f}}(k)$, define the orbit $$\Fl_{w}^\circ(\mathbf{f}', \mathbf{f}) := L^+ \mathcal{G}_{\mathbf{f}'} \cdot \tilde{w} \cdot e \subset \Fl_{\mathcal{G}},$$ where $\tilde{w} \in LG(k)$ is any lift of $w$ and $e$ is the basepoint of $\Fl_{\mathcal{G}_{\mathbf{f}}}$. Let $\Fl_{w}(\mathbf{f}', \mathbf{f})$ be the closure of $\Fl_{w}^\circ(\mathbf{f}', \mathbf{f})$. The stabilizer of $\tilde{w} \cdot e$ in $L^+ \mathcal{G}_{\mathbf{f}'}$ is the positive loop group of the parahoric $\mathcal{O}_{F}$-group scheme associated to $\mathbf{f}' \cup w \mathbf{f}$, cf. \cite[Proposition 3.7]{AGLR:Local}. It follows that $\Fl_{w}^\circ(\mathbf{f}', \mathbf{f})$ is canonically isomorphic to the perfection of a smooth quasi-projective $k$-scheme. Furthermore, $\Fl_{w}(\mathbf{f}', \mathbf{f})$ is the perfection of a projective $k$-scheme by \cite{BhattScholze:Projectivity} (if $F$ has characteristic zero) and \cite{PappasRapoport:Twisted} (if $F$ has characteristic $p$). We call $\Fl_{w}^\circ(\mathbf{f}', \mathbf{f})$ a \emph{perfect Schubert cell} and $\Fl_{w}(\mathbf{f}', \mathbf{f})$ a \emph{perfect Schubert scheme}.

Following \cite{HainesRapoport:Parahoric, AGLR:Local}, the perfect Schubert schemes can be parametrized as follows. Fix a maximal $\breve{F}$-split torus $S \subset G$ containing $A$ and let $T = C_G(S)$. Then $T$ is a maximal torus, and we obtain a chain of $F$-tori
$A \subset S \subset T.$
The connected N\'eron model $\mathcal{T}$ of $T$ over $\mathcal{O}_F$ is contained in $\mathcal{G}$ \cite[Proposition 8.2.4]{KalehtaPrasad:Bruhat}, and $\mathcal{T}(\mathcal{O}_{\breve{F}})$ is the unique parahoric subgroup of $T(\breve{F})$.
The Iwahori--Weyl group is
\begin{equation} \label{eqn-Iwahori-Weyl}
     W := N_G(S)(\breve{F})/ \mathcal{T}(\mathcal{O}_{\breve{F}}).
\end{equation}
Choose an alcove $\mathbf{a}$ in the apartment $\mathscr{A}(G,A,F)$. This determines a splitting $W = W_{\af} \rtimes \pi_1(G)_I$. Here $W_{\af} \subset W$ is the affine Weyl group, $I$ is the inertia group of $F$, and $\pi_1(G)$ is the algebraic fundamental group. As $W_{\af}$ is a Coxeter group, we may use this splitting to extend the length function $\ell$ and partial Bruhat order to $W$ by declaring elements of $\pi_1(G)_I$ to have length $0$. The Iwahori group scheme over $\mathcal{O}_F$ associated to the alcove $\mathbf{a}$ will be denoted by $\mathcal{I} : = \mathcal{G}_{\mathbf{a}}.$

The facets $\mathbf{f}, \mathbf{f}'$ are contained in unique facets $\breve{\mathbf{f}}, \breve{\mathbf{f}}'$ in $\mathscr{A}(G,\breve{F},S)$. Then $\mathcal{G}_{\breve{\mathbf{f}}} = \mathcal{G}_{\mathbf{f}} \times \Spec(\mathcal{O}_{\breve{F}})$.
Let $W_{\breve{\mathbf{f}}} = (N_G(S)(\breve{F}) \cap \mathcal{G}_{\mathbf{f}} (\mathcal{O}_{\breve{F}}))/\mathcal{T}(\mathcal{O}_{\breve{F}})$.
By the Bruhat decomposition (see \cite[Proposition 8]{HainesRapoport:Parahoric}),
$$W_{\breve{\mathbf{f}}'} \backslash W / W_{\breve{\mathbf{f}}}  = L^+ \mathcal{G}_{\mathbf{f}'}(\overline{k}) \backslash LG(\overline{k}) /  L^+ \mathcal{G}_{\mathbf{f}}(\overline{k}).$$
Let $\sigma \in \Gal(\overline{k}/k)$ be a geometric Frobenius element (if $\overline{k} = k$ then we let $\sigma$ be the identity) and let $W_{\mathbf{f}} :=  W_{\breve{\mathbf{f}}}^\sigma$.
By \cite[Remark 9]{HainesRapoport:Parahoric}, passing to $\sigma$-invariants induces a bijection
\begin{equation} \label{eqn-k-points}
    W_{\mathbf{f}'} \backslash W^\sigma / W_{\mathbf{f}}  = L^+ \mathcal{G}_{\mathbf{f}'}(k) \backslash LG(k) /  L^+ \mathcal{G}_{\mathbf{f}}(k).
\end{equation}

\subsubsection{Closure relations and dimensions}
From now on we assume that $\mathbf{f}$ and $\mathbf{f}'$ are contained in the closure of the alcove $\mathbf{a}$. By the reduction steps in \cite[\S 3.1]{HainesRicharz:Local}, over $\overline{k}$ every $\Fl_{w}(\mathbf{f}', \mathbf{f})$ is isomorphic to a perfect Schubert scheme for a pair $(\mathbf{f}', \mathbf{f})$ satisfying this assumption. 

\textit{For the remainder of Section \ref{sect-Geometric-results} we assume $k = \overline{k}$ unless otherwise stated.} This implies $F= \breve{F}$.
The results still apply when $k$ is finite if we restrict to perfect Schubert schemes defined over $k$.
In this subsection we recall one of the main results in \cite{Richarz:Schubert}. We note that the results in \textit{op.~cit.}~are stated in the case where $F$ has equal characteristic, but the proofs apply verbatim in mixed characteristic.

By \cite[Lemma 1.6]{Richarz:Schubert}, for $w \in W$ there is a unique element $w ^\mathbf{f}$ of minimal length in $w W_{\mathbf{f}}$, and a unique element $_{\mathbf{f}'} w ^\mathbf{f}$ of maximal length in $\{(vw)^\mathbf{f} \: : \: v \in W_{\mathbf{f}'}\}$. Taking equivalence classes gives a natural bijection
$$_{\mathbf{f}'} W ^\mathbf{f} := \{ _{\mathbf{f}'} w ^\mathbf{f} \: : \: w \in W\} \cong W_{\mathbf{f}'} \backslash W / W_{\mathbf{f}}.$$

\begin{lem} \label{lem-closure}
The perfect Schubert scheme $\Fl_{w}(\mathbf{f}', \mathbf{f})$ is set-theoretically the following union of locally closed perfect Schubert cells.
$$\Fl_{w}(\mathbf{f}', \mathbf{f}) = \bigsqcup_{\substack{v \in _{\mathbf{f}'} W ^\mathbf{f} \\ v \leq _{\mathbf{f}'} w ^\mathbf{f}}} \Fl_{v}^\circ(\mathbf{f}', \mathbf{f}).$$ The dimension of $\Fl_{w}(\mathbf{f}', \mathbf{f})$ is $\ell(_{\mathbf{f}'} w ^\mathbf{f})$.
\end{lem}

\begin{proof}
This is \cite[Proposition 2.8]{Richarz:Schubert}. When $\mathbf{f}=\mathbf{a}$, this is proved using the Demazure resolution \eqref{eqn-Demazure-def}. For general $\mathbf{f}$, one analyzes the fibers of $\Fl_{\mathcal{I}} \rightarrow \Fl_{\mathcal{G}}$ using that $L^+\mathcal{G}/L^+\mathcal{I}$ is the perfection of a flag variety $\overline{P}^{\text{red}}/\overline{B}$ (see for example \eqref{eqn-flag-variety}).
\end{proof}

\subsubsection{Deperfections}
An integral domain $A$ is said to be $p$-\emph{closed} if for every $a \in \Frac(A)$ such that $a^p \in A$, we have $a \in A$. An integral scheme is $p$-closed if it has a covering by open affines which are $p$-closed. We recall the following special case of the construction in \cite[Proposition A.15]{Zhu:Mixed}.
\begin{prop} \label{prop-deperfection-construction}
Let $X$ be a perfect scheme over $k$ which is isomorphic to the perfection of an integral projective $k$-scheme, and let $k(X)$ be the field of rational functions on $X$. 

Then for each subfield $L \subset k(X)$ which is finitely generated over $k$ and whose perfection is $k(X)$, there exists a unique $p$-closed\footnote{By the remark after Proposition 1 in \cite{Itoh-weak-normality}, $p$-closure is equivalent to weak normality.}  projective $k$-scheme $X'$ such that $k(X') = L$ and $(X')_{\perf} = X$.
\end{prop}

\begin{proof}
The scheme $X'$ is the ringed space with underlying topological space the same as $X$ and the sheaf of rings
$\mathcal{O}_{X'}(U) = \{f \in \mathcal{O}_X(U) \: : \: f \in L\}.$ By the proof of \cite[Proposition A.15]{Zhu:Mixed}, $X'$ is the unique $p$-closed scheme such that $k(X') = L$ and $(X')_{\perf} = X$. It remains to see that $X'$ projective. By assumption, $X$ has an ample line bundle $\mathcal{L}$. Because $X'$ is of finite-type, some power of $\mathcal{L}$ descends along the natural map $X \rightarrow X'$. By \cite[Lemma 3.6]{BhattScholze:Projectivity}, the line bundle on $X'$ is ample.
\end{proof}

In \cite[\S B.2]{Zhu:Mixed}, Zhu constructed canonical deperfections of Witt vector Schubert schemes when $\mathcal{G}$ is a split reductive group, and conjectured that these are normal and Cohen-Macaulay. Canonical deperfections for general parahorics were constructed in \cite{AGLR:Local}. In \cite[Theorem 1.10]{AGLR:Local}, it was shown that canonical deperfections occurring in the $\mu$-admissible locus of a minuscule cocharacter $\mu$ (with a minor technical assumption when $p=2$) satisfy Zhu's conjecture. We now show that \emph{every} $p$-closed deperfection of a perfect Schubert scheme is normal (this result is also implicit in \cite{AGLR:Local}).

\begin{lem} \label{Aperf.Normal}
Let $A$ be a domain of characteristic $p$ that is $p$-closed. Then if $A_{\perf}$ is normal, so is $A$.
\end{lem}

\begin{proof}
Let $f(x) \in A[x]$ be a monic polynomial and let $a \in \Frac(A)$ be an element such that $f(a) = 0.$ Then $a \in A_{\perf}$ by normality, so $a \in A$ because $A$ is $p$-closed.
\end{proof}

\begin{lem} \label{lem-perf-normal}
Every perfect Schubert scheme $\Fl_{w}(\mathbf{f}', \mathbf{f})$ is normal.
\end{lem}

\begin{proof}
This is stated in \cite[Proposition 3.7]{AGLR:Local} when $\mathbf{f} = \mathbf{f}'$ and $F$ has characteristic zero, but the proof applies verbatim to the general case. The key ingredients are the fact that the Demazure resolution $D_{\dot{w}}$ \eqref{eqn-Demazure-def} is normal, satisfies $\pi_{\dot{w}, *}(\mathcal{O}_{D_{\dot{w}}}) \cong \mathcal{O}_{\Fl_{w}(\mathbf{f}', \mathbf{f})}$, and the fact that the ring of global functions on a normal scheme is normal \cite[Tag 0358]{stacks-project}.
\end{proof}

\begin{theorem}
Every deperfection\footnote{We only consider $p$-closed deperfections in this paper, so we shall omit the adjective ``$p$-closed" from now on.} of $\Fl_{w}(\mathbf{f}', \mathbf{f})$ in the sense of Proposition \ref{prop-deperfection-construction} is normal.
\end{theorem}

\begin{proof}
Since every deperfection is $p$-closed, this follows from Lemmas \ref{Aperf.Normal} and \ref{lem-perf-normal}.
\end{proof}

\begin{remark}
If $F$ has characteristic $p$, $\Fl_{w}(\mathbf{f}', \mathbf{f})$ is isomorphic to the perfection of a canonical projective variety $\Fl_{w}^{\can}(\mathbf{f}', \mathbf{f})$. To construct $\Fl_{w}^{\can}(\mathbf{f}', \mathbf{f})$, form the affine flag variety as the \'etale quotient $(L_pG/L_p^+\mathcal{G})_{\et}$ on all $k$-algebras. This is an ind-projective $k$-scheme by \cite{PappasRapoport:Twisted}, and $\Fl_{w}^{\can}(\mathbf{f}', \mathbf{f})$ is the scheme-theoretic image of the $L^+_p\mathcal{G}$-orbit associated to $w$. If $p \nmid |\pi_1(G_{\text{der}})|$ (and the relative root system is reduced if $p =2$), 
 $\Fl_{w}^{\can}(\mathbf{f}', \mathbf{f})$ is normal by
 \cite[Theorem 4.23]{FHLR:Local} (see also 
 \cite[Theorem 8]{Faltings:Loop} and \cite[Theorem 0.3]{PappasRapoport:Twisted}). In general $\Fl_{w}^{\can}(\mathbf{f}', \mathbf{f})$ need not be normal by \cite[Theorem 1.1]{HainesLorencoRicharz:Remaining}, but the seminormalization of $\Fl_{w}^{\can}(\mathbf{f}', \mathbf{f})$ is normal by \cite[Theorem 4.1]{FHLR:Local}. The seminormalization  \cite[Tag 0EUK]{stacks-project} is the universal scheme mapping universally homeomorphically onto $\Fl_{w}^{\can}(\mathbf{f}', \mathbf{f})$ with the same residue fields. The deperfection of $\Fl_{w}(\mathbf{f}', \mathbf{f})$ associated by Proposition \ref{prop-deperfection-construction} to the function field of $\Fl_{w}^{\can}(\mathbf{f}', \mathbf{f})$ is the seminormalization of $\Fl_{w}^{\can}(\mathbf{f}', \mathbf{f})$.
\end{remark}

\subsection{Demazure resolutions and convolution Grassmannians}
\subsubsection{Demazure resolutions} \label{sect-Demazure}
Let $w \in W$ and write $w = w_{\af} \tau$ for $w_{\af} \in W_{\af}$ and $\tau \in \pi_1(G)_I$. Choose a (not necessarily reduced) decomposition $\dot{w} = s_1 \cdots s_n$ of $w_{\af}$ as a product of simple reflections along the walls of $\mathbf{a}$. For each $i$ let $\mathcal{P}_i$ be the minimal parahoric group scheme attached to $s_i$, so that $\mathcal{P}_i(\mathcal{O}_F) = \mathcal{I}(\mathcal{O}_F) \sqcup \mathcal{I}(\mathcal{O}_F) s_i \mathcal{I}(\mathcal{O}_F)$ for any lift of $s_i$ to $N_G(S)(F)$. 
The \textit{perfect Demazure scheme associated to} $\dot{w}$ is
\begin{equation} \label{eqn-Demazure-def}
    D_{\dot{w}} := L^+\mathcal{P}_1 \times^{L^+ \mathcal{I}} \cdots \times^{L^+ \mathcal{I}} L^+\mathcal{P}_n / L^+ \mathcal{I}.
\end{equation} Here we have taken \'etale quotient of $\prod_{i=1}^n L^+\mathcal{P}_i$ by the right action of $\prod_{i=1}^n L^+\mathcal{I}$ given by
$$(i_1, \cdots, i_n) \cdot (p_1, \cdots, p_n) = (p_1i_1, i_1^{-1} p_2 i_2, \cdots, i_{n-1}^{-1}p_n i_n ).$$ Since any lift of $\tau$ to $N_G(S)(F)$ normalizes $L^+\mathcal{I}$, the multiplication map $(p_1, \cdots, p_n) \mapsto p_1 \cdots p_n \tau$ induces a map $D_{\dot{w}} \rightarrow \Fl_{\mathcal{G}_{\mathbf{f}}}$. We define $\pi_{\dot{w}}(D_{\dot{w}})$ to be the image of this map, which is isomorphic to $\Fl_{w'}(\mathbf{a}, \mathbf{f})$ for some $w' \in W$. Let
\begin{equation}\label{eqn-Dem-res}
\pi_{\dot{w}} \colon D_{\dot{w}} \rightarrow \pi_{\dot{w}}(D_{\dot{w}})
\end{equation}
be the resulting surjective map.

The following result is well-known in equal characteristic \cite[Lemma 9]{Faltings:Loop}, \cite[Proposition 9.7]{PappasRapoport:Twisted} and in mixed-characteristic \cite[Proposition 3.8]{XuhuaZhou:ADLV}, \cite[Proposition 3.7]{AGLR:Local}. Below we assemble the arguments in one place for the reader's convenience.

\begin{lem}\label{lem-Demazure.rational}
(1) The perfect Demazure scheme $D_{\dot{w}}$ is the perfection of a smooth, projective $k$-scheme. The multiplication map $\pi_{\dot{w}}$ satisfies $$R\pi_{\dot{w}, *}(\mathcal{O}_{D_{\dot{w}}}) \cong \mathcal{O}_{\pi_{\dot{w}}(D_{\dot{w}})}[0] \quad \text{and} \quad R\pi_{\dot{w}, !}(\mathbb{F}_p) \cong \mathbb{F}_p[0].$$ 

\noindent (2) Furthermore, if $w \in {}_{\mathbf{f}'} W ^\mathbf{f}$ and $\dot{w}$ is a reduced expression, then $\pi_{\dot{w}}(D_{\dot{w}})= \Fl_{w}(\mathbf{f'}, \mathbf{f})$ and $\pi_{\dot{w}}$ is an isomorphism over $\Fl_{w}^\circ(\mathbf{f}', \mathbf{f})$.
\end{lem}

\begin{proof}
By the arguments in \cite[Proposition 8.7]{PappasRapoport:Twisted} (which apply verbatim in mixed characteristic), $L^n_p\mathcal{P}_i/ L_p^n \mathcal{I} \cong \mathbb{P}^1_k$ for $n \gg 0$, and this quotient admits sections Zariski-locally. As in \cite[Proposition 3.5]{XuhuaZhou:ADLV}
this implies that $D_{\dot{w}}$ is the perfection of an iterated $\mathbb{P}^1_k$-bundle which is smooth and projective. By \cite[Lemma 6.9]{BhattScholze:Projectivity}, to show that $R\pi_{\dot{w}, *}(\mathcal{O}_{D_{\dot{w}}}) \cong \mathcal{O}_{\pi_{\dot{w}}(D_{\dot{w}})}[0]$, it suffices to check this on geometric fibers. 

We first suppose $\mathbf{f} = \mathbf{f}' = \mathbf{a}$, and we induct on the length of $\dot{w}$. Let $\dot{v} = s_2 \cdots s_{n} \tau$. Since $\pi_{\dot{v}}(D_{\dot{v}})$ is irreducible and $L^+\mathcal{I}$-stable, it is of the form $\Fl_{v'}(\mathbf{a}, \mathbf{a})$ for some $v' \in W$. As in \cite[Lemma 9]{Faltings:Loop}, by performing the multiplication in two steps $(p_1, \cdots, p_n) \mapsto (p_1, p_2 \cdots p_n \tau) \mapsto p_1 \cdots p_n \tau$, we may factor $\pi_{\dot{w}}$ as
$$
D_{\dot{w}} \longrightarrow  L^+\mathcal{P}_1 \times^{L^+ \mathcal{I}}\Fl_{v'}(\mathbf{a}, \mathbf{a}) \longrightarrow \pi_{\dot{w}}(D_{\dot{w}}).
$$
By induction, pushforward along the first map preserves the structure sheaf. By properties of Tits systems \cite[(1.2.6)]{BruhatTits:I} (especially T 3), if $s_1 v' < v'$, then $\pi_{\dot{w}}(D_{\dot{w}}) = \Fl_{v'}(\mathbf{a}, \mathbf{a})$ and the second map has fibers isomorphic to $(\mathbb{P}^1)_{\perf}$. 
Otherwise, if $s_1 v' > v'$ then $\pi_{\dot{w}}(D_{\dot{w}}) = \Fl_{s_1v'}(\mathbf{a}, \mathbf{a})$. To describe the fibers in this case, consider the union $Z = \cup_{v'' < v'}  \Fl_{v''}(\mathbf{a}, \mathbf{a})$ of those perfect Schubert schemes in $\Fl_{v'}(\mathbf{a}, \mathbf{a})$ for which $s_1 v'' < v''$. Then the second map is an isomorphism away from $Z$ and has fibers isomorphic to $(\mathbb{P}^1)_{\perf}$ over $Z$.  Thus, in either case the result follows since $R\Gamma((\mathbb{P}^1_L)_{\perf}, \mathcal{O}_{(\mathbb{P}^1_L)_{\perf}}) = L[0]$ for any perfect field $L$.

For general $\mathbf{f}, \mathbf{f}'$, we factor $\pi_{\dot{w}}$ as
\begin{equation}
D_{\dot{w}} \longrightarrow  \Fl_{\mathcal{I}} \longrightarrow \Fl_{\mathcal{G}_{\mathbf{f}}},
\end{equation}
where the first map is a Demazure map with $\mathbf{f} = \mathbf{f}' = \mathbf{a}$, and the second map is the quotient. The image of the first map is of the form $\Fl_{w'}(\mathbf{a}, \mathbf{a})$ for some $w' \in W$. Thus, it suffices to show that pushforward along the natural surjection $\pi \colon \Fl_{w'}(\mathbf{a}, \mathbf{a}) \rightarrow \pi_{\dot{w}}({D_{\dot{w}}})$ preserves the structure sheaf. For this, let $\overline{P}^{\red}$ be the maximal reductive quotient of the special fiber of $\mathcal{G}_{\mathbf{f}}$. By \cite[Remark 2.9]{Richarz:Schubert}, the image of the special fiber of $\mathcal{I} \rightarrow \mathcal{G}_{\mathbf{f}}$ in $\overline{P}^{\red}$ is a Borel subgroup $\overline{B}$, and we have
\begin{equation} \label{eqn-flag-variety} L_p^+\mathcal{G}_{\mathbf{f}}/L_p^+\mathcal{I} \cong \overline{P}^{\red} / \overline{B}.
\end{equation}
Furthermore, the maximal $F$-split torus $S$ has a natural $\mathcal{O}_F$-structure whose special fiber $\overline{S}$ is a maximal torus in $\overline{P}^{\red}$. The group $W_{\mathbf{f}}$ identifies with the Weyl group of $(\overline{P}^{\red}, \overline{S})$. Thus, since $\pi$ is $L^+\mathcal{I}$-equivariant, the fibers of $\pi$ are isomorphic to perfections of unions of $\overline{B}$-orbit closures in $\overline{P}^{\red} / \overline{B}$. Such a union is connected, so by normality of $\pi_{\dot{w}}({D_{\dot{w}}})$  this implies that $\pi_{*}(\mathcal{O}_{\Fl_{w'}(\mathbf{a}, \mathbf{a})}) \cong \mathcal{O}_{\pi_{\dot{w}}({D_{\dot{w}}})}$. For the higher direct images, it is well-known that the structure sheaves of $\overline{B}$-orbit closures, i.e, Schubert varieties, have vanishing higher cohomology. (For example, this follows because the Demazure resolution is a rational resolution by an iterated $\mathbb{P}^1$-bundle \cite[Theorem 3.3.4]{BrionKumar:Frobenius}). If our fiber in $\overline{P}^{\red} / \overline{B}$ is not irreducible, we induct on its support in $\overline{P}^{\red} / \overline{B}$. Write an arbitrary union of $\overline{B}$-orbit closures in terms of its irreducible components as $X = X_1 \cup \cdots \cup X_m$. If $Y = X_2 \cup \cdots \cup X_m$, we conclude by induction and the exact sequence\footnote{Unions and intersections of Schubert varieties are reduced by \cite[Theorem 2.3.3]{BrionKumar:Frobenius}, but we do not need this fact when working with perfections.}
$$
0 \longrightarrow \mathcal{O}_{X_1 \cup Y} \longrightarrow \mathcal{O}_{X_1} \oplus \mathcal{O}_{Y} \longrightarrow \mathcal{O}_{X_1 \cap Y} \longrightarrow 0.
$$
This proves that $R\pi_{\dot{w}, *}(\mathcal{O}_{D_{\dot{w}}}) \cong \mathcal{O}_{\pi_{\dot{w}}(D_{\dot{w}})}$ for general $\mathbf{f}, \mathbf{f}'$.\footnote{A different proof for the step from $\mathbf{f} = \mathbf{a}$ to general $\mathbf{f}$ appears in the final part of the proof of \cite[Proposition 3.1 i)]{HainesRicharz:Local}.}

To compute $\pi_{\dot{w}, !}(\mathbb{F}_p)$, we first note that since $\pi_{\dot{w}}$ is the perfection of a proper map, we can replace $\pi_{\dot{w}, !}$ by $\pi_{\dot{w}, *}$. Then we apply $R\pi_{\dot{w}, *}$ to the Artin--Schreier sequence 
$$
0 \longrightarrow \mathbb{F}_p \longrightarrow \mathcal{O}_{D_{\dot{w}}} \xlongrightarrow{\Frob - \Id} \mathcal{O}_{D_{\dot{w}}} \longrightarrow 0
.$$

Finally, the proof of \cite[Proposition 2.8]{Richarz:Schubert} shows that
$$\Fl_{w}^\circ(\mathbf{f}', \mathbf{f}) = \Fl^\circ_{_{\mathbf{f}'} w ^\mathbf{f}}(\mathbf{a}, \mathbf{f}).$$ Thus, if $w \in {}_{\mathbf{f}'} W ^\mathbf{f}$ and $\dot{w}$ is a reduced expression, by tracing through our proof that $R\pi_{\dot{w}, *}(\mathcal{O}_{D_{\dot{w}}}) \cong \mathcal{O}_{\pi_{\dot{w}}(D_{\dot{w}})}$, it follows that $\pi_{\dot{w}}(D_{\dot{w}})= \Fl_{w}(\mathbf{f'}, \mathbf{f})$ and $\pi_{\dot{w}}$ and is an isomorphism over $\Fl_{w}^\circ(\mathbf{f}', \mathbf{f})$.
\end{proof}

\subsubsection{Convolution}
In this subsection we let $\mathcal{G} = \mathcal{G}_{\mathbf{f}}$ and $\Fl_w = \Fl_w(\mathbf{f}, \mathbf{f})$. The convolution diagram is as follows:
\begin{equation} \label{eqn-convolution-diagram}
\Fl_{\mathcal{G}} \times \Fl_{\mathcal{G}} \xlongleftarrow{p} LG \times \Fl_{\mathcal{G}} \xlongrightarrow{q} LG \times^{L^+ \mathcal{G}} \Fl_{\mathcal{G}} \xlongrightarrow{m} \Fl_{\mathcal{G}}.
\end{equation}
Here $p$ is the quotient map on the first factor and the identity on the second, $q$ is the quotient by the diagonal action of $L^+ \mathcal{G}$, and $m$ is the multiplication map. The map $LG \times^{L^+ \mathcal{G}} \Fl_{\mathcal{G}} \rightarrow \Fl_{\mathcal{G}} \times \Fl_{\mathcal{G}}$, $(g_1, g_2) \mapsto (g_1, g_1g_2)$ is an isomorphism, so the convolution Grassmannian $LG \times^{L^+ \mathcal{G}} \Fl_{\mathcal{G}}$ is represented by an increasing union of perfections of projective $k$-schemes.

\begin{lem} \label{lem-conv-points}
If $k$ is a finite field or an algebraic closure thereof and $x \in \Fl_{\mathcal{G}}(k)$,
$$m^{-1}(x) = \{(y, y^{-1}x) \: : \: y \in G(F)/K\} \subset G(F)/K \times G(F)/K.$$
\end{lem}

\begin{proof}
This follows from the isomorphism $LG \times^{L^+ \mathcal{G}} \Fl_{\mathcal{G}} \rightarrow \Fl_{\mathcal{G}} \times \Fl_{\mathcal{G}}$ and Lemma \ref{lem-k-points}, cf. also \cite[Lemma 5.6.1]{Zhu:Intro}.
\end{proof}

We now define the following finite-dimensional perfect subschemes \eqref{eqn-conv-def} of the convolution Grassmannian. Fix $w_1, w_2 \in W$. Let $n \gg 0 $ be an integer such that $L^+\mathcal{G}$ acts on $\Fl_{w_2}$ through the quotient $L^+ \mathcal{G} \rightarrow L^n\mathcal{G}$. Let $\pi^{(n)} \colon L^{(n)}G \rightarrow \Fl_{\mathcal{G}}$ be the reduction of the $L^+\mathcal{G}$-torsor $LG \rightarrow \Fl_{\mathcal{G}}$ to an $L^n\mathcal{G}$-torsor. Then we define
\begin{equation}\label{eqn-conv-def} \Fl_{w_1} \times^{L^+\mathcal{G}} \Fl_{w_2} : = (\pi^{(n)})^{-1}(\Fl_{w_1}) \times^{L^n\mathcal{G}} \Fl_{w_2}.\end{equation} This is a closed subscheme of $LG \times \Fl_{\mathcal{G}}$, isomorphic to the perfection of a projective $k$-scheme, and independent of $n \gg 0$.

\begin{theorem} \label{thm-convolution}
Let $w_1, w_2 \in {}_{\mathbf{f}} W ^\mathbf{f}$. Fix reduced expressions $\dot{w}_1 = \tau_1 s_1 \cdots s_n$ and $\dot{w}_2 = t_1 \cdots t_m \tau_2$ where the $s_i$, $t_j$ are simple reflections and $\tau_i \in \pi_1(G)_I$. Let $\dot{w} = s_1 \cdots s_n t_1 \cdots t_m \tau_2$.
\begin{enumerate}
    \item The image of the restriction of $m$ to $\Fl_{w_1} \times^{L^+\mathcal{G}} \Fl_{w_2}$ is the left $\tau_1$-translate $\tau_1 \cdot \pi_{\dot{w}}(D_{\dot{w}})$.
    \item The resulting convolution map
$$m \colon \Fl_{w_1} \times^{L^+\mathcal{G}} \Fl_{w_2} \longrightarrow \tau_1 \cdot \pi_{\dot{w}}(D_{\dot{w}})$$ satisfies
$Rm_*(\mathcal{O}_{\Fl_{w_1} \times^{L^+\mathcal{G}} \Fl_{w_2}}) \cong \mathcal{O}_{\tau_1 \cdot \pi_{\dot{w}}(D_{\dot{w}})}[0]$ and $Rm_!(\mathbb{F}_p) \cong \mathbb{F}_p[0]$.
\end{enumerate}
\end{theorem}

\begin{proof}
Any lift $\tilde \tau_1 \in N_G(S)(F)$ acts on $\Fl_{\mathcal{G}}$ by multiplication on the left, and this action preserves $L^+\mathcal{I}$-orbits, so $\tau_1 \cdot \pi_{\dot{w}}(D_{\dot{w}})$ is well-defined. Since $m$ is $LG$-equivariant, after multiplying on the left by $\tilde \tau_1$ we can assume $w_1 \in W_{\af}$. Now consider the following commutative diagram.
$$\xymatrix{
D_{\dot{w}} \ar@{=}[d] \ar[rr]^{\pi_{\dot{w}}} & & \pi_{\dot{w}}(D_{\dot{w}})  \\
D_{\dot{w}_1} \times^{L^+ \mathcal{I}} D_{\dot{w}_2}  \ar[r]^{\Id \times \pi_{\dot{w}_2}} & D_{\dot{w}_1} \times^{L^+ \mathcal{I}} \Fl_{w_2} \ar[r]^{\pi_{\dot{w}_1} \times \Id} &  \ar[u]^m \Fl_{w_1} \times^{L^+\mathcal{G}} \Fl_{w_2}
}$$
Here we have factored  $\pi_{\dot{w}}$ as the composition
$$(p_1, \cdots, p_n, q_1, \cdots, q_m) \mapsto (p_1, \cdots, p_n, q_1 \cdots q_m \tau_2) \mapsto (p_1 \cdots p_n, q_1 \cdots q_m\tau_2) \mapsto p_1 \cdots q_m\tau_2.$$
By Lemma \ref{lem-Demazure.rational}, each map (except possibly $m$) is surjective and pushforward preserves the structure sheaf. For the bottom two maps, we also use flat base change and the fact that both properties can be checked \'etale-locally. It follows that $m$ maps $\Fl_{w_1} \times^{L^+\mathcal{I}} \Fl_{w_2}$ onto $\pi_{\dot{w}}(D_{\dot{w}})$, and $Rm_*$ preserves the structure sheaf. It also follows that $Rm_!(\mathbb{F}_p) \cong \mathbb{F}_p[0]$.
\end{proof}

\subsection{Geometry of generalized Mirkovi\'c--Vilonen cycles}
\subsubsection{Attractors and fixed points} \label{sect-attractors}
Let $R$ be a ring, and let $\mathbb{G}_{m}$ be the multiplicative group over $\Spec(R)$. For an $R$-scheme $X$ with an action of $\mathbb{G}_m$, we have the following functors on $R$-algebras, cf. \cite{Drinfeld:Gm, Richarz:Gm}. 
\begin{align} \label{eqn-Gm-action}
    X^0 \colon R' \mapsto & \Hom_R^{\mathbb{G}_m}(\Spec(R'), X),  \\
    X^+ \colon R' \mapsto  & \Hom_R^{\mathbb{G}_m}(\mathbb{A}^1_{R'}, X). \label{eqn-Gm-action2}
\end{align}
Here $\Spec(R')$ has the trivial $\mathbb{G}_m$-action and $\mathbb{A}^1_{R'}$ has the usual (multiplicative) action. These are the functors of \emph{fixed-points} and \emph{attractors}, respectively (we will not need the repellers). There are natural maps
$$\xymatrix{
X^0 \ar@/^1pc/[rr]^{\iota^0} & X^+ \ar[l]^{q^+} \ar[r]_{\iota^+} & X.
}$$
Here $q^+$ is induced by the zero section $\Spec(R') \rightarrow \mathbb{A}^1_{R'}$. If $X$ is an (ind)-perfect $k$-scheme with an action of $(\mathbb{G}_m)_{\perf}$ we also use the notation $X^0$ and $X^+$ to denote the restriction of these functors to \emph{perfect} $k$-algebras. The domain of these functors will be clear from the context, and the their formation commutes with the passage from $X$ to $X_{\perf}$.

For the rest of subsection \ref{sect-attractors} we assume $R=k$ and that $X$ is a separated $k$-scheme of finite-type.

\begin{lem}\label{lem-attractor-props}

\noindent (1) The functors $X^0$ and $X^+$ are represented by $k$-schemes of finite-type. 

\noindent (2) The map $\iota^0$ is a closed immersion and $\iota^+$ is a monomorphism.

\noindent (3) The map $q^+$ is affine, has geometrically connected fibers, and induces a bijection $\pi_0(X^+) \cong \pi_0(X^0)$.
\end{lem}

\begin{proof}
This is \cite[1.2.2, 1.3.3, 1.4.2]{Drinfeld:Gm} and \cite[Corollary 1.12]{Richarz:Gm}.
\end{proof}

We write the connected components of $X^+$ as
$$X^+ = \bigsqcup_{i \in \pi_0(X^0)} X^+_i.$$
Similarly, let $\{X_i^0\}_{i \in \pi_0(X^0)}$ be the connected components of $X^0$. Let $q^+_i \colon X^+_i \rightarrow X^0_i$ be the restriction of $q^+$ to the corresponding connected component.

A sequence $Z_0 \subset Z_1 \subset \cdots \subset Z_n = X$ of closed subschemes is called a \emph{filtration} of $X$, and the schemes $Z_i \setminus Z_{i-1}$ are the \emph{cells} of the filtration. With additional assumptions, we can say more about $X^+$ as in the following lemma.

\begin{lem} \label{lem-X.normal.proj}
Suppose $k$ is algebraically closed and that $X$ is normal, projective, and connected. Then $q^+$ is  bijective on points and each $q_i^+$ is a locally closed immersion. Furthermore, there exists a filtration of $X$ having the $X_i^+$ as its cells.
\end{lem}

\begin{proof}
This is a special case of \cite[\S 2.3]{CP:Constant}. We mention that the hypotheses on $X$ are used to ensure that there exists a $\mathbb{G}_m$-equivariant embedding of $X$ into a projective space $\mathbb{P}(V)$, where $\mathbb{G}_m$ acts linearly on $V$. The existence of such an embedding is a result of Sumihiro \cite{Sumihiro:Equivariant}.~Since the $\mathbb{G}_m$-action on $V$ can be diagonalized, the filtration can be explicitly constructed as in the works of Bia\l ynicki-Birula \cite{BB:Actions, BB:Properties}.
\end{proof}

For applications to mod $p$ Hecke algebras, we will apply the following result to a deperfection of the Demazure resolution \eqref{eqn-Dem-res} in Proposition \ref{prop-equiv-res} when $F$ has characteristic $p$. When $F$ has characteristic zero we will verify directly that the conclusion of lemma is still true for the perfect Witt vector Demazure resolution.

\begin{lem} \label{lem-attractor.coh}
In addition to the assumptions of Lemma \ref{lem-X.normal.proj}, suppose moreover that there exists a smooth projective $k$-scheme $\tilde X$ equipped with a $\mathbb{G}_m$-action and a $\mathbb{G}_m$-equivariant surjection $\pi \colon \tilde X \rightarrow X$. \\
(1) There is exactly one $i_0 \in \pi_0(X^+)$ such that $X_{i_0}^+$ is closed in $X$, and it is characterized by the property that $q^+_{i_0}$ is a universal homeomorphism.\\
(2) Furthermore, if $R\pi_!(\mathbb{F}_p) = \mathbb{F}_p[0]$, we have
$$
Rq_{i!}^+(\mathbb{F}_p)=
\begin{cases}
\mathbb{F}_p[0], & i= i_0 \\
0, & \text{otherwise.}
\end{cases}$$
\end{lem}

\begin{proof}
This is \cite[Corollary 2.3.5]{CP:Constant}; we sketch the argument here. For (1), we note that by Lemma \ref{lem-X.normal.proj} there exists some $i_0 \in \pi_0(X^+)$ such that $X_{i_0}^+$ is closed in $X$. For general $i \in \pi_0(X^+)$ the fiber $\pi^{-1}(X_i)$ is a union of connected components of $\tilde{X}^+$. Since $\pi^{-1}(X_{i_0})$ is projective it contains a closed attractor by \cite{BB:Properties}, which is necessarily also closed in $\tilde{X}$. As $\tilde{X}$ is smooth and projective it has a unique closed attractor by \cite[\S 4]{BB:Actions}, so $i_0$ is uniquely determined. By Lemma \ref{lem-attractor-props} (3) the restriction of $q^+$ to any connected component of $X^+$ which is also proper maps bijectively onto its image in $X^0$. Since $q^+$ admits a section it follows that $i_0$ is uniquely characterized by property that $q^+_{i_0}$ is an isomorphism on reduced loci and hence is a universal homeomorphism.

For (2), $Rq_{{i_0}!}^+(\mathbb{F}_p)=\mathbb{F}_p[0]$ because $q^+_{i_0}$ is a universal homeomorphism. It remains to show that $Rq_{i!}^+(\mathbb{F}_p)=\mathbb{F}_p[0]$ for all other $i$. Let $\pi_i$ be the restriction of $\pi$ to $\pi^{-1}(X_i)$. For $j \in \pi_0((\pi^{-1}(X_i))^+)$ let $\tilde{q}^+_j \colon \tilde{X}_j^+ \rightarrow \tilde{X}_j^0$ be the restriction of $\tilde{q}^+ \colon \tilde{X}^+ \rightarrow \tilde X^0$ to $\tilde{X}_j^+$. Let $\tilde{\iota} \colon \tilde{X}^0 \rightarrow \tilde{X}$ be the inclusion.
By Lemma \ref{lem-X.normal.proj}, $\pi^{-1}(X_i)$ has a filtration with cells isomorphic to the $\tilde{X}_j^+$ (obtained by intersecting a filtration of $\tilde{X}$ with $\pi^{-1}(X_i)$). Since $R\pi_!(\mathbb{F}_p) = \mathbb{F}_p[0]$ then $Rq_{{i}!}^+(\mathbb{F}_p) = R\pi_{i!}( Rq_{{i!}}^+(\mathbb{F}_p))$. Hence, by excision the object $Rq_{{i}!}^+(\mathbb{F}_p) \in D_{\et}(X^0_i, \mathbb{F}_p)$ lies in the subcategory generated under taking cones of morphisms between the objects $R\pi_{!} \circ \tilde \iota_! \circ R\tilde q_{j!}^+ (\mathbb{F}_p)$ for $j \in \pi_0((\pi^{-1}(X_i))^+)$ (see also \cite[Corollary 2.3.5]{CP:Constant} for an argument using exact sequences). It therefore suffices to prove that $R\tilde q_{j!}^+ (\mathbb{F}_p) = 0$. For this, we note that because $\tilde{X}$ is smooth then $\tilde q_{j}^+$ is an affine bundle by \cite[\S 4]{BB:Actions}, and because $i \neq i_0$ then $\tilde q_{j}^+$ has positive relative dimension. The desired vanishing then follows from the fact that $R\Gamma_c(\mathbb{A}^n, \mathbb{F}_p) = 0$ if $n > 0$, which may be computed from excision and the quasi-coherent cohomology of $\mathbb{P}^n$, cf. \cite[Lemma 2.1.1]{CP:Constant}. 
\end{proof}

\subsubsection{Parabolic subgroups} \label{sect-parabolics} \

\textit{For the rest of the paper we assume $\mathbf{f} = \mathbf{f}'$. We let $\mathcal{G} = \mathcal{G}_{\mathbf{f}}$ and $\Fl_w = \Fl_w(\mathbf{f}, \mathbf{f})$.} 

Fix a semi-standard $F$-Levi subgroup $M \subset G$, i.e.~the centralizer of a subtorus of $A$. Let $\lambda \colon \mathbb{G}_m \rightarrow T$ be a cocharacter defined over $F$. By conjugation, this gives rise to a $\mathbb{G}_m$-action on $G$. Then the attractor $P:=G^+$ as in \eqref{eqn-Gm-action2} is a parabolic $F$-subgroup of $G$. We may choose $\lambda$ so that $G^0 = M$ is the fixed point functor as in \eqref{eqn-Gm-action} for this $\mathbb{G}_m$-action defined by $\lambda$. The $k$-points of the unipotent radical $U \subset P$ are
\begin{equation} \label{eqn-unipotent} U(k) = \{g \in G(k) \:: \: \lim_{t \longrightarrow 0} \lambda(t) \cdot g \cdot \lambda(t)^{-1} = 1\}. \end{equation}

\begin{lem} \label{lem-parahoric-levi} Let $\mathcal{P} := \mathcal{G}^+$ and $\mathcal{M} := \mathcal{G}^0$ for the unique extension of $\lambda$ to a cocharacter $\mathbb{G}_{m, \mathcal{O}_F} \rightarrow \mathcal{T}$.
Suppose $k$ is a finite field or an algebraic closure thereof. \\
(1) The attractor functor $\mathcal{P}$ and fixed-point functor $\mathcal{M}$ are smooth, closed subgroup schemes of $\mathcal{G}$ with geometrically connected fibers.\\
(2) Furthermore, $\mathcal{M}$ is a parahoric group scheme for $M$, the natural map $\Fl_{\mathcal{M}} \rightarrow \Fl_{\mathcal{G}}$ is a closed immersion, and $\mathcal{M}(\mathcal{O}_F) = \mathcal{G}(\mathcal{O}_F) \cap M(F).$
\end{lem}

\begin{proof}
The first part follows from \cite[Lemma 4.5, Lemma 4.6]{HainesRicharz:Test}. The proof that $\Fl_{\mathcal{M}} \rightarrow \Fl_{\mathcal{G}}$ is a closed immersion in equal characteristic applies verbatim in mixed characteristic, cf. \cite[Proposition 1.20]{Zhu:Mixed}. Finally, the description of $\mathcal{M}(\mathcal{O}_F)$ is \cite[Lemma A.1]{Richarz:Affine}.
\end{proof}

If $F$ has characteristic zero, the Teichm\"uller map induces a homomorphism $(\mathbb{G}_m)_{\perf} \rightarrow L^+\mathbb{G}_m$. The cocharacter $\lambda \colon \mathbb{G}_{m, \mathcal{O}_F} \rightarrow \mathcal{T}$ also induces a homomorphism $L^+\mathbb{G}_m \rightarrow L^+\mathcal{T}$. Since $\mathcal{T} \subset \mathcal{I}$ and $L^+\mathcal{I}$ acts on $\Fl_{\mathcal{G}}$,  this induces a $(\mathbb{G}_m)_{\perf}$-action on $\Fl_{\mathcal{G}}$. We may then define the fixed points $(\Fl_{\mathcal{G}})^0$ and attractors $(\Fl_{\mathcal{G}})^+$ as functors on perfect $k$-algebras. If has characteristic $p$ we replace the Teichm\"uller map with natural inclusion of units $R^{\times} = \mathbb{G}_m(R) \rightarrow \mathbb{G}_m(R[\![t]\!]) = L^+\mathbb{G}_m(R)$ for a $k$-algebra $R$.

\subsubsection{Equivariant resolutions}

\begin{theorem} Let $w \in {}_{\mathbf{f}} W ^\mathbf{f}$ and fix a reduced decomposition $\dot{w}$ as in \ref{sect-Demazure}. 
Then there exists a $\mathbb{G}_m$-equivariant deperfection \eqref{eqn-dem-dep} of the map $\pi_{\dot w} \colon D_{\dot w} \rightarrow \Fl_{w}$.
\end{theorem}

\begin{proof}
First assume $F$ has characteristic $p$. Then $D_{\dot{w}}$ has a canonical deperfection $$D_{\dot{w}}^{\dep} : = L^+\mathcal{P}_{1,p} \times^{L^+_p \mathcal{I}} \cdots \times^{L^+_p \mathcal{I}} L^+\mathcal{P}_{n,p} / L^+_p \mathcal{I}.$$ This is a smooth, projective, iterated $\mathbb{P}^1$-bundle. By considering the birational map $\pi_{\dot{w}}$, we can view $k(D_{\dot{w}}^{\dep})$ as a subfield of $k(\Fl_{w})$. By Proposition \ref{prop-deperfection-construction}, this gives us a particular deperfection $\Fl_{w}^{\dep}$ of $\Fl_{w}$, equipped with a birational map \begin{equation} \label{eqn-dem-dep} \pi_{\dot{w}}^{\dep} \colon D_{\dot{w}}^{\dep} \rightarrow \Fl_{w}^{\dep} \end{equation} which is a deperfection of \eqref{eqn-Dem-res}. In fact, $\Fl_{w}^{\dep}$ is the seminormalization of the usual afine Schubert variety associated to $w$ as in \cite{PappasRapoport:Twisted}. Moreover, $\pi_{\dot{w}}^{\dep}$ is $L_p^+ \mathcal{I}$-equivariant so it is also $\mathbb{G}_m$-equivariant for the action induced by \begin{equation} \label{gm-action-composition} \mathbb{G}_m \rightarrow L_p^+ \mathbb{G}_m \xrightarrow{L_p^+ \lambda} L_p^+ \mathcal{T} \rightarrow L^+_p \mathcal{I}.\end{equation}

If $F$ has characteristic zero, the maps ${L^+_p \mathcal{I}} \rightarrow L^+\mathcal{P}_{i,p}$ have non-reduced kernels, so we cannot define $D_{\dot{w}}^{\dep}$ as above (the following method also works if $F$ has characteristic $p$). Instead, let $n$ be an integer large enough so that $L^+\mathcal{I}$ acts on $D_{\dot w}$ and $\Fl_{w}$ through the quotient $L^n\mathcal{I}$. Let $H \subset L^n_p\mathcal{I}$ be the unique reduced closed subgroup whose perfection is the stabilizer of a chosen point in the open orbit in $D_{\dot{w}}$ (and hence also in $\Fl_{w}$ because $\pi_{\dot w}$ is an isomorphism over the open orbit). Then $H$ is a smooth affine $k$-group by \cite[Lemma A.26]{Zhu:Mixed}. Let $D_{\dot{w}}^{\dep}$ and $\Fl_{w}^{\dep}$ be the deperfections associated to $k(L^n_p\mathcal{I}/H)$ by Proposition \ref{prop-deperfection-construction}. Since these deperfections are normal they admit $L^n_p\mathcal{I}$-actions which extend the action over the open orbit, cf. \cite[Proposition 3.3]{AGLR:Local}. The map $\pi_{\dot w}$ also deperfects to an $L^n_p\mathcal{I}$-equivariant map as in \eqref{eqn-dem-dep} because $\pi_{\dot w}$ is $L^n\mathcal{I}$-equivariant and birational. Then we get the desired $\mathbb{G}_m$-action from the composition \eqref{gm-action-composition}.
\end{proof}

\begin{remark}
We do \emph{not} assert that $\Fl_{w}^{\dep}$ is the canonical deperfection of $\Fl_{w}$ in \cite{AGLR:Local}, which is constructed from the quotient of $L_p^+ \mathcal{G}$ by the stabilizer of its action on $\Fl_{w}^\circ$. 
\end{remark}

\begin{theorem} \label{prop-equiv-res}
The resolution constructed in \eqref{eqn-dem-dep} satisfies all the conclusions of Lemma \ref{lem-X.normal.proj} and Lemma \ref{lem-attractor.coh}.
\end{theorem}

\begin{proof}
If $F$ has characteristic $p$, \eqref{eqn-dem-dep} satisfies all the hypotheses of these lemmas. Here we use topological invariance of the \'etale site \cite[04DY]{stacks-project} which implies that $R\pi_{\dot{w} ,!}^{\dep}(\mathbb{F}_p) \cong \mathbb{F}_p[0]$ can be checked after passing to perfections, so it follows from Lemma \ref{lem-Demazure.rational}. 

If $F$ has characteristic zero, \eqref{eqn-dem-dep} satisfies all the hypotheses of Lemma \ref{lem-X.normal.proj}, but it is possible that $D_{\dot{w}}^{\dep}$ is not smooth. However, the proof of Lemma \ref{lem-attractor.coh} applies verbatim provided that we verify:

\begin{enumerate}
    \item There is a unique closed attractor in $(D_{\dot{w}}^{\dep})^+$.
    \item For each non-closed attractor $\tilde{X}_i^+ \subset (D_{\dot{w}}^{\dep})^+$, the map $\tilde{q}_i^+ \colon \tilde{X}_i^+ \rightarrow \tilde{X}_i^0$ satisfies $R\tilde q_{i!}^+ (\mathbb{F}_p) \cong \mathbb{F}_p[0].$
\end{enumerate}
Both of these properties may be checked after passing to the $(\mathbb{G}_m)_{\perf}$-action on $D_{\dot{w}}$. We proceed by induction on the length of $w$. We may assume $w \in W_{\af}$. Write $\dot{w} = s_1 \ldots s_n$. Then $D_{\dot{s_i}} \cong (\mathbb{P}^1)_{\perf}$ and we have a decomposition $D_{\dot{s_i}} = \Spec(k) \sqcup \mathbb{A}^1_{\perf}$ where $\Spec(k)$ is the $k$-point corresponding to the identity coset $e \in \mathcal{P}_{s_i}(\mathcal{O}_F)/ \mathcal{I}(\mathcal{O}_F)$ and $\mathbb{A}^1_{\perf}(k) = \mathcal{I}(\mathcal{O}_F) s_i \mathcal{I}(\mathcal{O}_F)/ \mathcal{I}(\mathcal{O}_F)$. We always have $\{e\}$, $\{s_i\} \in (D_{\dot{s}_i})^0$. There are three possibilities, depending on whether $(\mathbb{G}_m)_{\perf}$ fixes, attracts, or repels the affine root group corresponding to $s_i$, respectively:
\begin{itemize}
    \item We have $(D_{\dot{s}_i})^0 = D_{\dot{s}_i}$, or
    \item $(\mathbb{G}_m)_{\perf}$ attracts $\mathbb{A}^1_{\perf} - \{s_i\}$ toward $\{s_i\}$, or
    \item $(\mathbb{G}_m)_{\perf}$ repels $\mathbb{A}^1_{\perf} - \{s_i\}$ toward $\{e\}$.
\end{itemize}
In the last two cases, one attractor is a point and the other is an $(\mathbb{A}^1)_{\perf}$-bundle over its fixed point.
Since $R\Gamma_c(\mathbb{A}^1_{\perf}, \mathbb{F}_p) =0$, this implies that (1) and (2) are satisfied when $n=1$. For the inductive step, let $t= s_1 \cdots s_{n-1}$ and let
$\pi \colon D_{\dot{w}} = D_{\dot{t}} \times^{L^+\mathcal{I}} D_{\dot{s}_n} \rightarrow D_{\dot{t}}$ be the projection. If $(D_{\dot{s}_n})^0 = D_{\dot{s}_n}$ then $\pi^{-1}$ induces a bijection between $\pi_0((D_{\dot{t}})^0)$ and $\pi_0((D_{\dot{w}})^0)$. Furthermore, in this case $\pi$ identifies fibers of the map $(D_{\dot{w}})^+ \rightarrow (D_{\dot{w}})^0$ with fibers of the  map $(D_{\dot{t}})^+ \rightarrow (D_{\dot{t}})^0$. Thus, by induction we are done in the first case.

In the other two cases, for each connected component $\tilde{Y}_i^+ \subset (D_{\dot{t}})^+$, $\pi^{-1}(\tilde{Y}_i^+)$ is a union of two connected components $\tilde{X}_{i1}^+$, $\tilde{X}_{i2}^+$ of $(D_{\dot{w}})^+$. The fixed points $\tilde{X}_{i1}^0$ and $\tilde{X}_{i2}^0$ each map isomorphically under $\pi$ onto $\tilde{Y}_i^0$. One of the two connected components, say $\tilde{X}_{i1}^+$, maps isomorphically under $\pi$ onto its image $\tilde{Y}_i^+$ in $(D_{\dot{t}})^+$. The other connected component $\tilde{X}_{i2}^+$ has the structure of an $(\mathbb{A}^1)_{\perf}$-bundle over $\tilde{Y}_i^+$ coming from the non-closed component of  $(D_{\dot{s}_n})^+$. The unique closed attractor in $D_{\dot{w}}$ is the connected component $\tilde{X}_{i_01}^+$ corresponding to the unique closed attractor $\tilde{Y}_{i_0}^+ \subset (D_{\dot{t}})^+$. All other $\tilde{X}_{i1}$ satisfy $Rq_{i1, !}^+(\mathbb{F}_p) = 0$ by induction. Finally, $Rq_{i2, !}^+(\mathbb{F}_p) = 0$ for all $i$ because the map $q_{i2}^+$ factors through an $(\mathbb{A}^1)_{\perf}$-bundle (by forgetting the last factor  $D_{\dot{s}_n}$), cf.~the last part of the proof of Lemma \ref{lem-attractor.coh}. 
\end{proof}

\begin{remark}
The induction in the proof of Theorem \ref{prop-equiv-res} can be used to show that $D_{\dot{w}}$ has a filtrable decomposition by the connected components of $D_{\dot{w}}^+$ without appealing to the existence of a $\mathbb{G}_m$-equivariant deperfection or \cite{BB:Properties}. Moreover, the induction also applies in equal characteristic. Thus, in equal characteristic there are two proofs of Theorem \ref{prop-equiv-res} (and hence of the main results in this paper), one which appeals to abstract generalities on Bia\l ynicki-Birula maps, and one which uses the explicit nature of $D_{\dot{w}}$ as much as possible. 
\end{remark}

\subsubsection{Description of the attractors}
Since the partial order on a Coxeter group is directed, we may write
$\Fl_{\mathcal{G}} = \colim \Fl_w,$ where the colimit is over $w \in {}_{\mathbf{f}}  W ^\mathbf{f}$ and the transition maps are closed immersions. The formation of fixed points and attractors is compatible with filtered colimits of closed immersions \cite[Theorem 2.1]{HainesRicharz:Test}. Thus,
$(\Fl_G)^0 = \colim (\Fl_w)^0$ and $(\Fl_G)^+ = \colim (\Fl_w)^+.$
Here $(\Fl_G)^0$ is a closed sub-ind-scheme of $\Fl_{\mathcal{G}}$ and $(\Fl_G)^+$ is a disjoint union of locally closed sub-ind-schemes. The natural map $(\Fl_G)^+ \rightarrow (\Fl_G)^0$ induces a bijection $\pi_0((\Fl_G)^+) \cong \pi_0((\Fl_G)^0)$.

By the definition of $\mathcal{M}:=\mathcal{G}^0$, the closed immersion $\Fl_{\mathcal{M}} \rightarrow \Fl_{\mathcal{G}}$ factors through $(\Fl_{\mathcal{G}})^0$.  By the proof of \cite[Theorem 5.2]{AGLR:Local}, the map $\Fl_{\mathcal{M}} \rightarrow (\Fl_{\mathcal{G}})^0$ identifies $\Fl_{\mathcal{M}}$ with a disjoint union of connected components of $(\Fl_{\mathcal{G}})^0$ (see also \cite[Proposition 4.7, Lemma 4.11]{HainesRicharz:Test}). Furthermore, by \cite[Theorem 5.2]{AGLR:Local}, the connected components of $(\Fl_{\mathcal{G}})^0$ can be described as follows. Let $W_{M}$ be the Iwahori--Weyl group of $M$ and let $W_{M, \af}$ be its affine Weyl group. Then the inclusion $\pi_0(\Fl_{\mathcal{M}}) \hookrightarrow \pi_0((\Fl_{\mathcal{G}})^0)$ identifies with the inclusion $$W_{M, \af} \backslash W_M   \hookrightarrow W_{M,\af} \backslash W / W_{\mathbf{f}}.$$ 

For $c \in W_{M,\af} \backslash W / W_{\mathbf{f}}$, let $S_c$ be the corresponding connected component\footnote{Recall that we assume $k$ is algebraically closed, so $S_c$ is in fact geometrically connected.} of $(\Fl_{\mathcal{G}})^+$. This is a locally closed sub-ind-scheme of $\Fl_{\mathcal{G}}$. Then
$$\Fl_{\mathcal{G}} = \bigsqcup_{c \in W_{M,\af} \backslash W / W_{\mathbf{f}}} S_c.$$
The irreducible components of the $S_c \cap \Fl_w$ generalize\footnote{In \cite[Theorem 3.2]{MirkovicVilonen:Geometric}, the group $G$ is split and $M$ is a maximal torus.} the \textit{Mirkovi\'c--Vilonen cycles} in \cite[Theorem 3.2]{MirkovicVilonen:Geometric}.
The $k$-points of $S_c$ can be described as follows. Let $M_{\SC}$ be the simply connected cover of the derived group of $M$ and let $P_{\SC} = M_{\SC} \ltimes U$. The following lemma also appears in the proof of \cite[Theorem 5.2]{AGLR:Local}.

\begin{lem} \label{lem:semi-infinite-points}
For $c \in W_{M,\af} \backslash W / W_{\mathbf{f}}$, let $\tilde{c} \in \Fl_{\mathcal{G}}(k)$ be the image of a representative under the embedding $W / W_{\mathbf{f}} \hookrightarrow \Fl_{\mathcal{G}}(k)$. Then
$S_c(k) = P_{\SC}(F) \cdot \tilde{c}.$
\end{lem}

\begin{proof}
This follows from Lemma \ref{lem-k-points} and the definitions in \ref{sect-parabolics}.
\end{proof}

If $c \in W_{M, \af} \backslash W_M = \pi_1(M)_I$, let $(\Fl_{\mathcal{M}})^c$ be the corresponding connected component of $\Fl_{\mathcal{M}}$. Let $$\pi_c \colon S_c \rightarrow (\Fl_{\mathcal{M}})^c$$ be the resulting map obtained by restriction from the map $(\Fl_{\mathcal{G}})^+ \rightarrow (\Fl_{\mathcal{G}})^0$.

\begin{lem} \label{lem-satake-k-points}
Let $m \in \Fl_{\mathcal{M}}(k)$ and let $\tilde{m} \in M(F)$ be a representative. If $m$ lies in $(\Fl_{\mathcal{M}})^c$, then 
$$\pi_c^{-1}(m) = \{[\tilde{m} \cdot u] \: : \: u \in U(F)/(U(F) \cap K)\} \subset G(F)/K.$$
In the above, $[\tilde{m} \cdot u]$ is the equivalence class of $\tilde{m} \cdot u$ in $\Fl_{\mathcal{G}}(k)$ for $u \in U(F)$, and $[\tilde{m} \cdot u_1] = [\tilde{m} \cdot u_2]$ if and only if $u_1 \cdot (U(F) \cap K) = u_2 \cdot (U(F) \cap K)$.
\end{lem}

\begin{proof}
This is a consequence of \eqref{eqn-unipotent} and Lemma \ref{lem:semi-infinite-points} (see also \cite[Lemma 4.11]{HainesRicharz:Test}).
\end{proof}

\subsubsection{Attractors over finite fields} In this subsection we assume that $k$ is finite, and we fix an element $w \in W_{\mathbf{f}} \backslash W^\sigma / W_{\mathbf{f}}$.
By \eqref{eqn-k-points}, this gives rise to a perfect Schubert scheme $\Fl_{w}$ defined over $k$ which admits a $k$-point. Also, fix a connected component $(\Fl_{\mathcal{M}})^c$ of $\Fl_{\mathcal{M}}$ which is defined over $k$ and admits a $k$-point (and is therefore geometrically connected). Such connected components of $\Fl_{\mathcal{M}}$ are indexed by $c \in \pi_1(M)_I^\sigma$.

The ind-schemes $(\Fl_{\mathcal{G}})^0$ and $(\Fl_{\mathcal{G}})^+$ are defined over $k$, as are the maps $\Fl_{\mathcal{M}} \rightarrow (\Fl_{\mathcal{G}})^0 \rightarrow \Fl_{\mathcal{G}}$. Let $S_c$ be the geometrically connected component of $(\Fl_{\mathcal{G}})^+$ which maps to $(\Fl_{\mathcal{M}})^c$ under the map $(\Fl_{\mathcal{G}})^+ \rightarrow (\Fl_{\mathcal{G}})^0$.

\begin{lem} \label{lem-geom-conn}
Let $c \in \pi_1(M)_I^\sigma$ and $w \in W_{\mathbf{f}} \backslash W^{\sigma} / W_{\mathbf{f}}$.
Then $(\Fl_{\mathcal{M}})^c \cap \Fl_{w}$ is geometrically connected.
\end{lem}

\begin{proof}
Over $\overline{k}$, the intersection is a closed union of $L^+\mathcal{M}_{\mathcal{O}_{\breve{F}}}$-orbits inside a single connected component $(\Fl_{\mathcal{M}, \overline{k}})^c$ of $\Fl_{\mathcal{M}, \overline{k}}$. Every such union is connected because it contains the unique minimal Iwahori-orbit in $(\Fl_{\mathcal{M}, \overline{k}})^c$.
\end{proof}

Let $\pi_{c,w} \colon S_c \cap \Fl_{w} \rightarrow (\Fl_{\mathcal{M}})^c \cap \Fl_w$ be the restriction of the map $\pi_c \colon S_c \rightarrow (\Fl_{\mathcal{M}})^c$.

\begin{theorem} \label{thm-satake-k-points} Let $c \in \pi_1(M)_I^\sigma$ and $w \in W_{\mathbf{f}} \backslash W^{\sigma} / W_{\mathbf{f}}$. Let $m$ be a $k$-point of $(\Fl_{\mathcal{M}})^c \cap \Fl_w$ and let $\tilde m \in M(F)$ be a representative.

\noindent (1) At the level of $k$-points, we have
$$\pi_{c,w}^{-1}(m) = \Fl_w(k) \cap \{[\tilde m \cdot u] \: : \: u \in U(F)/(U(F) \cap K)\}.$$
In the above, $[\tilde{m} \cdot u]$ is the equivalence class of $\tilde{m} \cdot u$ in $\Fl_{\mathcal{G}}(k)$ for $u \in U(F)$, and $[\tilde{m} \cdot u_1] = [\tilde{m} \cdot u_2]$ if and only if $u_1 \cdot (U(F) \cap K) = u_2 \cdot (U(F) \cap K)$. \\
(2) Furthermore, we have
$$
R(\pi_{c,w})_{!}(\mathbb{F}_p)=
\begin{cases}
\mathbb{F}_p[0], & (\Fl_{\mathcal{M}, \overline{k}})^c \cap \Fl_{w, \overline{k}} \text{ is the unique closed attractor in } (\Fl_{w, \overline{k}})^+ \\
0, & \text{otherwise.}
\end{cases}$$
\end{theorem}

\begin{proof}
Lemma \ref{lem-satake-k-points} is valid if $k$ is finite, so this gives the description of $\pi_{c,w}^{-1}(m)$. To check the formula for $R(\pi_{c,w})_{!}(\mathbb{F}_p)$, we may base-change to $\overline{k}$. By Theorem \ref{prop-equiv-res}, we may then apply the conclusions of Lemma \ref{lem-attractor.coh}. It remains to see that $(\Fl_{\mathcal{M}, \overline{k}})^c \cap \Fl_{w, \overline{k}}$, which is a priori a disjoint union of connected components in $(\Fl_{w, \overline{k}})^0$, is connected. This follows from Lemma \ref{lem-geom-conn}.
\end{proof}

\begin{remark}
If $\mathbf{f}$ is a special vertex, $\Fl_{\mathcal{M}}(k) = (\Fl_{\mathcal{G}})^0(k)$ by the Iwasawa decomposition. If $\mathbf{f}$ is not very special, i.e, if $\mathbf{f}$ is not special over $\breve{F}$, then $\Fl_{\mathcal{M}}(\overline{k}) \subsetneq (\Fl_{\mathcal{G}})^0(\overline{k})$ in general. 
\end{remark}

\section{Applications to mod \texorpdfstring{$p$}{p} Hecke algebras} \label{sect-Hecke-alegbra}
\textit{In Section \ref{sect-Hecke-alegbra} we assume that $k = \mathbb{F}_q$ is a finite field of characteristic $p>0$.}
\subsection{Mod \texorpdfstring{$p$}{p} Hecke algebras}
\subsubsection{The function-sheaf dictionary}
In this subsection we recall the function-sheaf dictionary following \cite[Rapport]{SGA4.5}. Let $X$ be a separated scheme of finite-type over $k$. Let $\mathcal{F}$ be a constructible sheaf of $\mathbb{F}_p$-vector spaces on $X$ (we will only need the case where $\mathcal{F}$ is a constant sheaf). For each point $x \in X(k)$, the stalk $\mathcal{F}_x$ is a finite-dimensional representation of $\Gal(\overline{k}/k)$. Taking the trace of the action of the geometric Frobenius element $\gamma \in \Gal(\overline{k}/k)$ gives a value $\Tr(\gamma, \mathcal{F}_x) \in \mathbb{F}_p$. In this way, we get a function
\begin{equation} \label{eqn-function-sheaf} \mathcal{F}^{\Tr} \colon X(k) \rightarrow \mathbb{F}_p \quad x \mapsto \Tr(\gamma, \mathcal{F}_x).\end{equation}
On the other hand, if $X_{\overline{k}} = X \times \Spec(\overline{k})$ there is a natural action of $\Gal(\overline{k}/k)$ on the \'etale cohomology with compact support $H_c^i(X_{\overline{k}}, \mathcal{F})$. The function-sheaf dictionary is the following relationship between these actions.
\begin{lem} \label{lem-function-sheaf}
We have
$$\sum_{x \in X(k)} \Tr(\gamma, \mathcal{F}_x) = \sum_i (-1)^i \Tr(\gamma, H_c^i(X_{\overline{k}}, \mathcal{F})).$$ In particular, if $\mathcal{F} = \mathbb{F}_p$ is the constant sheaf,
$$|X(k)| \equiv  \sum_i (-1)^i \Tr(\gamma, H_c^i(X_{\overline{k}}, \mathbb{F}_p)) \pmod{p}.$$
\end{lem}

\begin{proof}
This follows from \cite[Rapport, Theorem 4.1]{SGA4.5}. 
\end{proof}

\begin{example}
If $X = \mathbb{P}^1$ and $\mathcal{F} = \mathbb{F}_p$, the left side of Lemma \ref{lem-function-sheaf} is $|\mathbb{P}^1(\mathbb{F}_q)| = q+1$. On the other hand, $R\Gamma(\mathbb{P}^1_{\overline{k}}, \mathbb{F}_p) = \mathbb{F}_p[0]$ so the right side of Lemma \ref{lem-function-sheaf} is $\Tr(\gamma, H^0(\mathbb{P}^1_{\overline{k}}, \mathbb{F}_p)) = 1 \equiv q+1 \pmod{p}$ since $\Gal(\overline{k}/ k)$ acts trivially on  $H^0(\mathbb{P}^1_{\overline{k}}, \mathbb{F}_p)$. If $X= \mathbb{A}^1$ then $R\Gamma_c(\mathbb{A}^1_{\overline{k}}, \mathbb{F}_p) = 0$, which agrees with the fact that $|\mathbb{A}^1(\mathbb{F}_q)| = q \equiv 0 \pmod{p}$.
\end{example}

\subsubsection{Definitions and bases}
The mod $p$ Hecke algebra associated to $K$ is the algebra $\mathcal{H}_{K}$ of compactly supported functions $K \backslash G(F) / K \rightarrow \mathbb{F}_p$. The multiplication is convolution:
\begin{equation} \label{eqn-convolution}
    (f_1 * f_2)(g) = \sum_{h \in G(F)/K} f_1(h) f_2(h^{-1}g), \quad f_1, f_2 \in \mathcal{H}_K, \: g \in G(F).
\end{equation}
The algebra $\mathcal{H}_{K}$ has a basis consisting of the characteristic functions of the double cosets $K \backslash G(F) / K$. By \eqref{eqn-k-points}, these double cosets are indexed by $ w \in W_{\mathbf{f}} \backslash W^\sigma / W_{\mathbf{f}}$. Let $\mathbbm{1}_{w}$ be the characteristic function of the corresponding double coset.

There is another basis of $\mathcal{H}_K$ more suitable for geometric arguments. For $w \in W_{\mathbf{f}} \backslash W^\sigma / W_{\mathbf{f}}$, let $\phi_{w} = \sum_{v} \mathbbm{1}_v$, where the sum runs over those $v \in W_{\mathbf{f}} \backslash W^\sigma / W_{\mathbf{f}}$ such that $_{\mathbf{f}} v ^\mathbf{f} \leq {}_{\mathbf{f}} w ^\mathbf{f}$. The basis $\phi_{w}$ arises from the function-sheaf dictionary as follows. Let $(\underline{\mathbb{F}}_p)_w$ be the constant \'etale sheaf supported on $\Fl_w$. Then by the closure relations in Lemma \ref{lem-closure}, we have
\begin{equation} \label{eqn-basis-Tr} ((\underline{\mathbb{F}}_p)_w)^{\Tr} = \phi_w.
\end{equation}

\begin{remark}
In the case of split groups in equal characteristic, $(\underline{\mathbb{F}}_p)_w$ is a shift of the mod $p$ intersection cohomology sheaf of $\Fl_w$ (by the perfected version of \cite[Theorem 1.5]{Cass:Perverse}). By \cite[Theorem 1.7]{Cass:Perverse}, the same is true in mixed characteristic if $\Fl_w$ has a deperfection by a scheme with $F$-rational singularities. It is an interesting question whether such a deperfection exists for all $w$ (for example, one could ask whether or not $\Fl_{w}^{\dep}$ is $F$-rational).
\end{remark}

\subsubsection{The Satake transform}
Recall the Levi decomposition $P = MU$ arising from the cocharacter $\lambda$. The Satake transform with respect to $P$ is the following map.
\begin{equation} \label{eqn-Satake.transform}
    \mathcal{S} \colon \mathcal{H}_{\mathcal{G}} \rightarrow \mathcal{H}_{M(F) \cap K}, \quad \mathcal{S}(f)(m) = \sum_{u \in U(F) / U(F) \cap K} f(mu), \: \: m \in M(F).
\end{equation}

\subsection{Explicit formulas}
Let $w_1, w_2 \in W_{\mathbf{f}} \backslash W^\sigma / W_{\mathbf{f}}$. The image of the convolution map $m \colon \Fl_{w_1} \times^{L^+\mathcal{I}} \Fl_{w_2} \rightarrow \Fl_{\mathcal{G}}$ is geometrically irreducible, closed, $L^+\mathcal{G}$-stable, and it is defined over $k$. Thus, it is of the form $\Fl_w$ for some $w \in  W_{\mathbf{f}} \backslash W^\sigma / W_{\mathbf{f}}$. Moreover, by Theorem \ref{thm-convolution}, $\Fl_{w, \overline{k}}$ is the image of a certain Demazure map, and $w$ may be computed explicitly by tracing through the steps in the proof of Lemma \ref{lem-Demazure.rational}.

\begin{theorem} \label{thm-conv-formula}
Let $w_1, w_2 \in W_{\mathbf{f}} \backslash W^\sigma / W_{\mathbf{f}}$, and let $w \in  W_{\mathbf{f}} \backslash W^{\sigma} / W_{\mathbf{f}}$ be the element such that $m(\Fl_{w_1} \times^{L^+\mathcal{G}} \Fl_{w_2}) = \Fl_w$.
Then
$$\phi_{w_1} * \phi_{w_2} = \phi_w.$$
\end{theorem}

\begin{proof}
By Lemma \ref{lem-conv-points}, for each $x \in \Fl_{\mathcal{G}}(k)$ we have
$$m^{-1}(x) = \{(y, y^{-1}x) \: : \: y \in \Fl_{w_1}(k), \: y^{-1}x \in \Fl_{w_2}(k)\}.$$ Recall that $((\underline{\mathbb{F}}_p)_{w_i})^{\Tr} = \phi_{w_i}$ by \eqref{eqn-basis-Tr}. Thus, for any $\tilde{x} \in G(F)$ a lift of $x$,
$$(\phi_{w_1} * \phi_{w_2})(\tilde{x}) = |m^{-1}(x)| \pmod{p}.$$
On the other hand, $Rm_!(\mathbb{F}_p) \cong (\underline{\mathbb{F}}_p)_w$ by Theorem \ref{thm-convolution}. Note that $m^{-1}(x)$ is nonempty if and only if $x \in \Fl_w(k)$, in which case it follows that $R\Gamma_c(m^{-1}(x), \mathbb{F}_p) = \mathbb{F}_p[0]$ where we view $m^{-1}(x)$ as a $k$-scheme. 
Now we conclude by applying the function-sheaf dictionary, using that $\sum_i (-1)^i \Tr(\gamma, H_c^i(m^{-1}(x)_{\overline{k}}, \mathbb{F}_p)) = \Tr(\gamma, H_c^0(m^{-1}(x)_{\overline{k}}, \mathbb{F}_p)) = 1$. Here we use the topological invariance of the \'etale site \cite[Tag 04DY]{stacks-project} to conclude that Lemma \ref{lem-function-sheaf} is also valid for the perfection of a finite-type $k$-scheme.
\end{proof}

For $c \in \pi_1(M)_I^\sigma$, let
$\phi_{c,w} \in \mathcal{H}_{M(F) \cap K}$ be the characteristic function of the intersection $(\Fl_{\mathcal{M}})^c(\mathbb{F}_q) \cap \Fl_w(\mathbb{F}_q)$. Recall that $M(F) \cap K = \mathcal{M}(\mathcal{O}_F)$ by Lemma \ref{lem-parahoric-levi}.

\begin{theorem} \label{thm-satake-formula}
The Satake transform $\mathcal{S} \colon \mathcal{H}_{\mathcal{G}} \rightarrow \mathcal{H}_{M(F) \cap K}$ satisfies
$$
\mathcal{S}(\phi_w)=
\begin{cases}
\phi_{c,w}, & (\Fl_{\mathcal{M}, \overline{k}})^c \cap \Fl_{w, \overline{k}} \text{ is the unique closed attractor in } (\Fl_{w, \overline{k}})^+ \\
0, & \text{the closed attractor in } (\Fl_{w, \overline{k}})^+ \text{does not admit a $k$-point in } \Fl_{\mathcal{M}}.
\end{cases}$$
\end{theorem}

\begin{proof}
Thanks to Theorem \ref{thm-satake-k-points}, the proof is analogous to that of Theorem \ref{thm-conv-formula}.
\end{proof}

\addtocontents{toc}{\protect\setcounter{tocdepth}{0}}
\subsection{The case of a special parahoric}
Throughout this section we assume that $M$ is a minimal Levi subgroup and that $\mathbf{f}$ is a special vertex. Since we assumed that $M$ is semi-standard, this implies that $M = C_G(A)$. By \cite[\S 6]{HenniartVigneras:Satake}, the group $\Lambda := M(F)/(M(F) \cap K)$ is a finitely generated abelian group. The Hecke algebra $\mathcal{H}_{M(F) \cap K}$ is naturally identified with the group algebra $\mathbb{F}_p[\Lambda]$. We denote the canonical basis elements by $e^{z}$ for $z \in \Lambda$.

The relative Weyl group $W(G,A)$ acts on $\Lambda$. Furthermore, the map
$$\Lambda \rightarrow K \backslash G(F) /K, \quad z \mapsto K z K$$
induces a bijection between $W(G,A)$-orbits in $\Lambda$ and the double cosets $K \backslash G(F) / K$. 
To describe these orbits, we note that there is a natural homomorphism of groups $\nu_M \colon M(F) \rightarrow \Hom(X^*(M)^{\Gal(\overline{F}/F)}, \mathbb{Z})$ characterized by the requirement
$$\nu_M(m)(\chi) = \text{val}_F(\chi(m)), \quad \text{for all }m \in M(F), \: \chi \in X^*(M)^{\Gal(\overline{F}/F)}.$$ Here $\text{val}_F$ is the normalized valuation on $F$. The restriction map $X^*(M)^{\Gal(\overline{F}/F)} \rightarrow X^*(A)$ induces an isomorphism $X_*(A) \otimes \mathbb{Q} \cong  \Hom(X^*(M)^{\Gal(\overline{F}/F)}, \mathbb{Z}) \otimes \mathbb{Q}$, and hence we may view $\nu_{\mathcal{M}}$ as a homomorphism
$$\nu_{M} \colon M(F) \rightarrow X_*(A) \otimes \mathbb{Q}.$$ The map $\nu_{M}$ is $W(G,A)$-equivariant and identifies $\Lambda/\Lambda_{\text{tor}}$ with a lattice in $X_*(A) \otimes \mathbb{Q}$. 

The minimal parabolic $P$ containing $M$ corresponds to a choice of positive roots $\Phi^+$ for $G$ in $X^*(A)$. This determines an anti-dominant Weyl chamber $\mathcal{C} \subset X_*(A) \otimes \mathbb{Q}$, defined by the requirement that $\langle \alpha, \nu \rangle \leq 0$ for all $\alpha \in \Phi^+$ and $\nu \in \mathcal{C}$. Let $\Lambda_-$ be the preimage of $\mathcal{C}$ in $\Lambda$. We call the elements of $\Lambda_-$ \emph{anti-dominant}. Each $W(G,A)$-orbit in $\Lambda$ contains a unique anti-dominant element, so we get bijections
$$
\Lambda_- \xlongrightarrow{\sim}  K \backslash G(F) / K \xlongleftarrow{\sim} W_{\mathbf{f}} \backslash W^\sigma / W_{\mathbf{f}}.
$$
For $z \in \Lambda_-$, let $\Fl_z \subset \Fl_{\mathcal{G}}$ be the corresponding perfect Schubert scheme, and let $\phi_z = ((\underline{\mathbb{F}}_p)_z)^{\Tr}$ be corresponding basis element of $H_K$. 

The following theorem recovers \cite[Theorem 1.2]{Herzig:Satake} (where $K$ is hyperspecial), and \cite[\S 1.5]{HenniartVigneras:Satake} (where $K$ is special). Moreover, we recover the explicit formulas in \cite[Proposition 5.1]{Herzig:Classification} and \cite[Theorem 5.5]{Ollivier:Inverse} (where $G$ is split and $K$ is hyperspecial), and in \cite[Theorem 1.1]{InverseSatake} (where $G$ is arbitrary and $K$ is special).

\begin{theorem} \label{thm-special-satake}
(1) If $K$ is special, for $z_1, z_2 \in \Lambda_-$, we have $$m(\Fl_{z_1} \times^{L^+\mathcal{G}} \Fl_{z_2}) = \Fl_{z_1z_2} \quad \text{and} \quad \phi_{z_1} * \phi_{z_2} = \phi_{z_1z_2}.$$
In particular, $\mathcal{H}_{K}$ is commutative. \\
(2) Moreover, the Satake transform $\mathcal{S} \colon  \mathcal{H}_{K} \rightarrow \mathcal{H}_{M(F) \cap K}$ is injective, identifies $\mathcal{H}_{K}$ with $\mathbb{F}_p[\Lambda_-]$, and satisfies
$$\mathcal{S}(\phi_z) = e^{z}, \quad \text{for all } z \in \Lambda_-.$$
\end{theorem}
\begin{proof}
By Theorem \ref{thm-conv-formula} and the fact that $\mathcal{S}$ is an algebra homomorphism, all of the statements follow from the formula $\mathcal{S}(\phi_z) = e^{z}$. To prove the formula, let $(\Fl_{\mathcal{M}})^c$ be the unique connected component $\Fl_{\mathcal{M}}$ which contains the $k$-point $z \in \Lambda_-$, and let $S_c \subset \Fl_{\mathcal{G}}$ be the corresponding attractor. Then we have the map $\pi_{c,z} \colon S_c \cap \Fl_{z} \rightarrow (\Fl_{\mathcal{M}})^c \cap \Fl_z$. Since affine Schubert schemes are $(\mathbb{G}_m)_{\perf}$-stable, the closure relations (Lemma \ref{lem-closure}) imply that $\pi_{c,z}^{-1}((\Fl_{\mathcal{M}})^c \cap \Fl_z^\circ) \subset \Fl_z^\circ$.
By \eqref{eqn-special-intersection} and Theorem \ref{thm-satake-k-points}, at the level of $k$-points,  $\pi_{c,z}^{-1}(\{z\}) = \{z\}.$
Then by the function-sheaf dictionary, $R(\pi_{c,z})_! (\mathbb{F}_p) \neq 0$. Thus, $(\Fl_{\mathcal{M}, \overline{k}})^c \cap \Fl_{z, \overline{k}}$ is the unique closed attractor in $(\Fl_{z, \overline{k}})^+$. On the other hand, $\Lambda$ is naturally identified with the $\sigma$-stable connected components of $\Fl_{\mathcal{M},\overline{k}}$ by \cite[Proposition 1.0.2]{HainesRostami:Satake}. Thus $z$ is the only $k$-point of $(\Fl_{\mathcal{M}})^c$, so the formula follows from Theorem \ref{thm-satake-formula}.
\end{proof}

\begin{remark}
A special vertex $\mathbf{f}$ is called \emph{very special} if it remains special in $\mathscr{B}(G,\breve{F})$. If $\mathbf{f}$ is very special, Theorem \ref{thm-special-satake} takes the following simpler form. First, by \cite[Lemma 6.1]{Zhu:Ramified}, a special vertex exists if and only if $G$ is quasi-split over $F$. In this case, $M=T$ is a maximal torus and $P = B$ is an $F$-rational Borel. Furthermore, $\Lambda = X_*(T)_I^\sigma$ by \cite[Corollary 11.1.2]{HainesRostami:Satake}. Over $\breve{F}$, we have a splitting $W = X_*(T)_I \rtimes W_0$ where $W_0 = N_G(S)(\breve{F})/T(\breve{F})$. Then ${W}_{\breve{\mathbf{f}}} = W_0$ and
${W}_{\breve{\mathbf{f}}} \backslash W / W_{\breve{\mathbf{f}}}$ and consists of the image of the anti-dominant elements $X_*(T)_-$, defined with respect to $B$, under the surjection $X_*(T) \rightarrow X_*(T)_I$ \cite[Corollary 1.8]{Richarz:Schubert}. Furthermore, the length $\ell$ is additive on these anti-dominant elements, so all convolution maps for $\Fl_{\mathcal{G}}$ are birational.
\end{remark}
\addtocontents{toc}{\protect\setcounter{tocdepth}{2}}
\bibliographystyle{amsalpha}
\bibliography{bibfile}

\end{document}